\documentclass{amsproc}

\usepackage{amsmath,amsfonts,amsthm,amssymb,graphicx,xcolor}

\usepackage[all,2cell,ps]{xy}

\theoremstyle{definition}

\theoremstyle{plain}
\newtheorem{theorem}[equation]{Theorem}
\newtheorem{corollary}[equation]{Corollary}
\newtheorem{lemma}[equation]{Lemma}
\newtheorem{proposition}[equation]{Proposition}
\newtheorem{conjecture}[equation]{Conjecture}
\newtheorem{define}[equation]{Definition}
\newtheorem{definitions}[equation]{Definitions}

\newtheorem{remark}[equation]{Remark}
\newtheorem{example}[equation]{Example}

\newcommand{\IC}{\mathbb{C}}

\newcommand{\IH}{\mathbb{H}}

\newcommand{\IR}{\mathbb{R}}

\newcommand{\IZ}{\mathbb{Z}}

\newcommand{\inj}{\rm inj}

\begin{document}

 \title[Singer Conjecture and Generalized Graph Manifolds]{Generalized Graph Manifolds, Residual Finiteness, and the Singer Conjecture}


\author{Luca F. Di Cerbo}
\address{Department of Mathematics, University of Florida, Gainesville, FL 32607, USA}
\curraddr{}
\email{dicerbo@ufl.edu}
\thanks{Supported in part by NSF grant DMS-2104662}

\author{Michael Hull }
\address{Department of Mathematics \& Statistics, University of North Carolina, Greensboro, NC 27402, USA}
\curraddr{}
\email{mbhull@uncg.edu}
\thanks{}


\date{}

\begin{abstract}
We use Price type inequalities to give a geometric proof of the Singer conjecture for extended graph manifolds and pure complex-hyperbolic higher graph manifolds with residually finite fundamental groups. In real dimension three, where a result of Hempel ensures that the fundamental group is always residually finite, we obtain a new proof of a well-known result of Lott and L\"uck. Finally, we give several classes of higher graph manifolds which do indeed have residually finite fundamental groups.
\end{abstract}

\maketitle



\tableofcontents


\section{Introduction}\label{introduction}

In order to state the \emph{Singer conjecture}, we need to define $L^2$-Betti numbers. $L^{2}$-Betti numbers are non-negative \emph{real} valued numerical invariants associated to a closed Riemannian manifold $M$ and an infinite regular cover $\widetilde{M}$ of such. They were originally introduced by Atiyah in \cite{Atiyah} in connection with $L^{2}$-index theorems. We briefly recall their definition. For an extensive treatment of $L^2$-Betti numbers and the technical details, we refer to L\"uck's treatise \cite{LuckBook}. 

Let $\Gamma$ be a discrete group, and let $M$ be a compact free proper $\Gamma$-manifold without boundary endowed with a $\Gamma$-invariant Riemannian metric. In other words, one has $M=\widetilde{M}\backslash \Gamma$, where $\Gamma\leq \text{Iso}(\widetilde{M})$ with $\widetilde{M}$ an infinite cover  of $M$, e.g., the \emph{universal} cover or the \emph{universal abelian} cover of $M$. Define the space of smooth $L^2$-integrable harmonic  $k$-forms 
\[
\mathcal{H}^k_{(2)} (\widetilde{M}) \; = \; \big\{ \, \omega\in\Omega^k (\widetilde{M}) \;  | \;  \Delta_d \omega = 0, \; \int_{\widetilde{M}} \omega \wedge *\omega  <\infty \big\}
\]
where $*$ is the Hodge star operator and $\Delta_d = d d^* + d^* d$ is the Hodge-Laplacian operator.
By \cite[Section 1.3.2]{LuckBook}, the spaces $\mathcal{H}^k_{(2)}(\widetilde{M})$ are finitely generated Hilbert modules over the von Neumann algebra $\mathcal{N}(\Gamma)$ of  $\Gamma$.  We define the $L^2$\emph{-Betti numbers}  $b^{(2)}_k \big(M; \mathcal{N}(\Gamma) \big)$ of $(M,\Gamma)$ as  the von Neumann dimension of the $\mathcal{N}(\Gamma)$-modules   $\mathcal{H}^k_{(2)}(\widetilde{M})$:
\[
b^{(2)}_k \big(M;\mathcal{N}(\Gamma) \big) \; \stackrel{{\rm def}}{=} \; \dim_{\mathcal{N}(\Gamma)}\mathcal{H}^k_{(2)}( \widetilde{M} ).
\] 
The $L^2$-Betti numbers assume values in the extended interval $[0,\infty]$ of the real numbers, and $b^{(2)}_k \big(M, \mathcal{N}(\Gamma) \big) \in [0,\infty)$ if the action of $\Gamma$ is co-compact. 

For simplicity's sake, when $\pi \colon \widetilde{M}\rightarrow M$ is the topological universal cover of $M$, we will refer to $L^2$-Betti numbers as follows:
\[
b^{(2)}_k(M) \quad \stackrel{{\rm def}}{=} \quad b^{(2)}_k \big(M;\mathcal{N}(\Gamma) \big).
\]  
As in this paper we are only concerned about $L^2$-Betti numbers computed with respect to the \emph{universal cover}, we believe this concise notation should not cause any confusion. Finally, we notice that
\begin{align}\label{G-dimension}
b^{(2)}_k(M)=0 \quad \Longleftrightarrow \quad \mathcal{H}^{k}_{2}(\widetilde{M})=0;
\end{align}
for a proof see Atiyah \cite{Atiyah} or \cite[Chapter I]{LuckBook}.

We can now state the celebrated and long-standing conjecture of Singer.

\begin{conjecture}[Singer Conjecture]\label{Singer}
	If $M$ is a closed aspherical manifold of real dimension $2n$, then
	\begin{equation*}
	b^{(2)}_{l}(M)=\begin{cases} 
	(-1)^n \chi_{\rm top}(M)  &\mbox{if }\quad    l =  n  \\
	0 & \mbox{if }\quad   l  \neq n.
	\end{cases}
	\end{equation*}
	If $\dim_{\IR}(M)=2n+1$, we have
	\begin{align}\notag
	b^{(2)}_l(M)=0, \mbox{   for any  } l.
	\end{align}
\end{conjecture}

Surprisingly enough, this conjecture is still open even under the assumption that $M$ admits a metric with \emph{strictly negative} sectional curvature! We refer to \cite{DS17} for more on the state of the art concerning this problem, and also for some evidence supporting this conjecture under weaker assumptions than negative sectional curvature.

The goal of this paper is to give a geometric proof of the Singer conjecture for two interesting classes of aspherical manifolds: \emph{extended graph manifolds} and \emph{pure complex-hyperbolic higher graph manifolds} with residually finite fundamental group. The class of extended graph manifolds was identified and studied in details by Frigerio-Lafont-Sisto in \cite{LafontBook} (specifically see Definition 0.2 in \cite{LafontBook}). These manifolds are decomposed into finitely many pieces (the vertices), each vertex is a manifold with boundary tori (the edges), and the tori boundary component of various pieces are glued together via affine diffeomorphisms. Moreover, the interior of each vertex is diffeomorphic either to a \emph{complete} finite-volume real-hyperbolic manifold with toral cusps, or to the product of a standard torus with a lower dimensional \emph{complete} finite-volume real-hyperbolic manifold with toral cusps. We denote by $(\IH^n_{\IR}, g_{\IR})$ the real-hyperbolic $n$-dimensional space with the hyperbolic metric normalized to have sectional curvature equal to $-1$.

\begin{theorem}\label{real}
	Let $M$ be an extended graph manifold with $k\geq 1$ real-hyperbolic pieces $\{(V_j, g_{\IR})\}^k_{j=1}$ and residually finite fundamental group. If $\dim_{\IR}(M)=2n$, we have
	\begin{equation*}
	b^{(2)}_{l}(M)=\begin{cases} 
	(-1)^n \chi_{\rm top}(M)=(-1)^n\sum^k_{j=1}\chi_{\rm top}(V_j)  &\mbox{if }\quad    l =  n,  \\
	0 & \mbox{if }\quad   l  \neq n.
	\end{cases}
	\end{equation*}	
    Finally, if $\dim_{\IR}(M)=2n+1$ we have
	\begin{align}\notag
	b^{(2)}_l(M)=0 \mbox{   for any  } l.
	\end{align}
\end{theorem}

The class of higher graph manifolds was first defined and studied by Connell-Su\'arez-Serrato in \cite{CSS19} (see Definition 1 in \cite{CSS19}). This class includes as sub-class the extended graph manifolds defined by Frigerio-Lafont-Sisto, see Remark 1 in \cite{CSS19}. Roughly speaking, the authors consider graph-like manifolds where the interior of the vertices are diffeomorphic to finite-volume locally symmetric spaces of rank one. Recall that such spaces have finitely many cuspidal ends with nilmanifolds cross-sections. The nilmanifold boundaries are then glued together via affine diffeomorphisms in order to manufacture a closed connected graph-like manifold. Here we are concerned with higher graph manifolds whose vertices are modeled on complex-hyperbolic manifolds with truncated cusps. We refer to such manifolds as pure complex-hyperbolic higher graph manifolds.  We denote by $(\IH^n_{\IC}, g_{\IC})$ the complex-hyperbolic $n$-dimensional space with the hyperbolic metric normalized to have holomorphic sectional curvature equal to $-1$.

\begin{theorem}\label{complex}
	Let $M$ be a pure complex-hyperbolic higher graph manifold with residually finite fundamental group. For $k\geq 2$, let $\{(W_j, g_{\IC})\}^k_{j=1}$ denote the complex-hyperbolic pieces. We the have
	\begin{equation*}
	b^{(2)}_{l}(M)=\begin{cases} 
	(-1)^n \chi_{\rm top}(M)=(-1)^n\sum^k_{i=1}\chi_{\rm top}(W_j)  &\mbox{if }\quad    l =  n,  \\
	0 & \mbox{if }\quad   l  \neq n.
	\end{cases}
	\end{equation*}	
\end{theorem}

\begin{remark}
	Because of the finite volume generalization of the Gauss-Bonnet formula (\emph{cf.} Harder \cite{Har71} and Cheeger-Gromov \cite{CG85}), we know that for any real-hyperbolic piece $(V_j, g_{\IR})$
	\begin{align}\label{harderr}
	\chi(V_j)=c(n)Vol_{g_{\IR}}(V_j),
	\end{align}
	for some non-zero constant $c(n)$ depending on the dimension only, and where $Vol_{g_{\IR}}(V_j)$ is the volume of $V_j$ computed with respect to the normalized real-hyperbolic metric. Similarly, in the complex hyperbolic case we have
	\begin{align}\label{harderc}
	\chi(W_j)=d(n)Vol_{g_{\IC}}(W_j),
	\end{align}
	for some non-zero constant $d(n)$ depending on the dimension only, and where $Vol_{g_{\IC}}(W_j)$ is the volume of $W_j$ computed with respect to the normalized complex-hyperbolic metric. By using Equations \eqref{harderr} and \eqref{harderc}, we can alternatively express the $L^2$-Betti numbers in the middle dimension of Theorem \ref{real} and Theorem \ref{complex}  in terms of the hyperbolic volume of the complex-hyperbolic pieces. 
\end{remark}

There are compelling reasons to believe that Theorem \ref{real} resolves an interesting particular case of the Singer conjecture. Moreover, we argue the general strategy of the proofs we present in this paper may play an important role towards the resolution of this conjecture for non-positively curved aspherical manifolds. Indeed, as proved by Leeb and Scott in \cite{Leeb}, any non-positively curved closed Riemannian manifold can be canonically decomposed along finitely many totally-geodesically embedded flat closed submanifolds of codimension one. The resulting pieces are either Seifert-fibered or codimension-one atoroidal (see \cite{Leeb} for the precise definitions). In other words, a general non-positively curved closed Riemannian manifold has a geometric decomposition very similar to an extended graph manifold. It is then conceivable that by equipping such manifolds with a sequence of Riemannian metrics that marginalizes the Seifert pieces and maximizes the codimension-one atroidal pieces (as we do for extended graph manifolds!), and by studying its pull-back to finite covers in conjunction with Price type arguments, one could approach the general case of the Singer conjecture for non-positively curved Riemannian manifolds. The proof of Theorem \ref{real} nicely outlines this new approach towards the Singer conjecture.

From the perspective of group theory, if $M$ is an extended graph manifold or higher graph manifold then $\pi_1(M)$ splits as a graph of groups where the vertex groups are fundamental groups of the pieces and the edge groups are the fundamental groups of the boundary components of the pieces which are not boundary components of $M$. Fern\'os-Valette used a Meyer-Vietoris argument to show if $G$ splits as a graph of groups with infinite edge groups and each vertex group $G_v$ has $b^{(2)}_1(G_v)=0$, then $b^{(2)}_1(G)=0$ \cite[Theorem 1]{FV17}. Together with the vanishing of the first $L^2$-Betti numbers of the pieces of $M$ (see \cite[Chapter 12]{LuckBook}), \cite[Theorem 1]{FV17} implies that  $b^{(2)}_1(M)=b^{(2)}_1(\pi_1(M))=0$. Using these methods, it is also possible to give algebraic proofs of Theorems \ref{real} and \ref{complex}.
 
We note that whenever a group $G$ splits as a graph of groups a formula for $b^{(2)}_1(G)$ in terms of the vertex and edge groups of the splitting is given in \cite[Theorem 1.5]{CHK20}. Also, if $\pi_1(M)$ satisfies Agol's RFRS condition (a strong form of residual finiteness), then the vanishing of $b^{(2)}_k(\pi_1(M))$ for $1\leq k\leq n$ is equivalent to the existence of finite index subgroup $\pi_1(M)$ which admits a homomorphism onto $\mathbb Z$ with the kernel of type $FP_n(\mathbb Q)$ \cite{F21}.

It is unknown whether all extended graph manifolds or all higher graph manifolds have residually finite fundamental group. See Section 12.1 in \cite{LafontBook}. We show that there are many examples which do have residually finite fundamental group, both in the extended graph manifold setting and in the pure complex-hyperbolic higher graph manifold setting.

First, we observe that a result of  Huang-Wise \cite{HuaW} implies the residual finiteness for extended graph manifolds where the pieces are all hyperbolic with \emph{virtually special} fundamental groups in the sense of Wise \cite{Wis21}. This virtually special machinery was the key to Agol's solution of the virtual Haken conjecture for 3-manifolds \cite{Ago13}, and it implies strong separability properties which are well-suited for the construction of finite covers.  Moreover, Bergeron-Haglund-Wise showed that arithmetic real-hyperbolic manifolds have virtually special fundamental group \cite{BHW11}. Hence we obtain the following.

\begin{theorem}\label{arithmetic}
	Let $M$ be an extended graph manifold with pieces $\{(V_j, g_{\IR})\}^k_{j=1}$ such that each $V_j$ is an arithmetic real-hyperbolic manifold. If $\dim_{\IR}(M)=2n$, we have
	\begin{equation*}
	b^{(2)}_{l}(M)=\begin{cases} 
	(-1)^n \chi_{\rm top}(M)=(-1)^n\sum^k_{j=1}\chi_{\rm top}(V_j)  &\mbox{if }\quad    l =  n,  \\
	0 & \mbox{if }\quad   l  \neq n.
	\end{cases}
	\end{equation*}	
    Finally, if $\dim_{\IR}(M)=2n+1$ we have
	\begin{align}\notag
	b^{(2)}_l(M)=0, \mbox{   for any  } l.
	\end{align}
\end{theorem}

Unfortunately, the virtually special machinery does not apply in the complex-hyperbolic setting. In this setting we use a different approach based on Hempel's proof of the residual finiteness of 3-manifold groups. Using this approach we present a procedure called \emph{iterated (twisted) partial doubles} which, starting with any finite volume real-hyperbolic or complex-hyperbolic manifold with at least two cusps, will construct an infinite family of distinct closed pure real-hyperbolic extended graph manifolds or pure complex-hyperbolic higher graph manifolds respectively with residually finite fundamental group.

Finally, in the last section we show that Theorem \ref{real} can be used to give a new proof of the result of  L\"uck and Lott that the Singer conjecture holds for 3-manifolds. As in \cite{LL95}, this is proved assuming the Geometrization conjecture.


\vspace{0.5 in}

\noindent\textbf{Acknowledgments}. LFDC thanks M. Anderson, G. Besson, and S. Maillot for introducing him to the theory of geometric $3$-manifolds during his time as a graduate student. He also thanks M. Stern for a stimulating ongoing collaboration on Price inequalities for harmonic forms, and for pointing out a mistake in a previous version of this paper. MH thanks D. Groves, D. B. McReynolds, and E. Einstein for helpful conversations about separability and virtually special groups. MH also thanks the University of Florida, which he visited during the preparation of this work, for its hospitality and support. Finally, the authors thank C. L\"oh for pointing out the reference \cite{Loh}, and the referees for pertinent comments on the paper.

\section{Extended Graph Manifolds and Pure Complex-Hyperbolic Higher Graph Manifolds}

In this section, we review the definition and basic properties of extended graph manifolds and higher graph manifolds.

Fix a dimension $n\geq 3$ for the underlying manifold $M^n$, an integer $k\geq 1$ giving the numbers of geometric pieces, and integers $2\leq n_i\leq n$ for any $i=1, ..., k$. For every $i$ let $V_i$ be a complete finite-volume real-hyperbolic manifold of dimension $n_i$ with toral cusps. Truncate the cusps of $V_i$ to obtain a compact manifold with torus boundaries say $\bar{V}_i$. For any $i=1, ..., k$ we define the $n$-dimensional manifold with boundary $N_i:=\bar{V}_i\times T^{n-n_i}$, where $T^{n-n_i}$ is the standard $(n-n_i)$-dimensional torus. Notice that the torus boundaries of $N_i$ have a well-defined affine structure. 
 
\begin{define}\label{geometric}[\emph{cf}. Definition 0.2. in \cite{LafontBook}]
	For $n\geq 3$ an extended graph $n$-manifold is a closed manifold built from $k\geq 1$ pieces $N_i$ as above by gluing their torus boundaries by affine diffeomorphisms.  An extended graph $n$-manifold is called \emph{pure real-hyperbolic} if all of the geometric pieces $N_i$ are real-hyperbolic (i.e., $N_i=\bar{V}_i$ for any $i=1, ..., k$).
\end{define}


It is clear that just like graph 3-manifolds, such manifolds can be nicely described by a graph. Each vertex is one of the geometric pieces $V_i$, and each edge is the collar of the torus boundary that is obtained by gluing two geometric pieces via an affine diffeomorphism.

\begin{define}[Section 2 in \cite{Hempel}, Definitions 2.12 \& 2.13 in \cite{LafontBook}]\label{graphofspaces}
Let $\Gamma$ be a graph with vertex set $V$ and edge set $E$. Let $\{X_v\}_{v\in V}$ and $\{X_e\}_{e\in E}$ be a collection of topological spaces and let $f_{e, v}\colon X_e\to X_v$ be a $\pi_1$-injective embedding whenever $v$ is an endpoint of $e$. Let $X$ be space constructed by gluing the copy of $X_e$ in $X_v$ to the copy of $X_e$ in $X_u$ via $f_{e, v}^{-1}\circ f_{e, u}$ whenever $u$ and $v$ are the endpoints of $e$ in $\Gamma$. Note that we may have $u=v$ but with distinct attaching maps. The space $X$ is a \emph{graph of spaces} with underlying graph $\Gamma$. We will often identify the vertex spaces $X_v$ and the edge spaces $X_e$ with their image in $X$. Without changing the homotopy type of $X$, we can also assume that each edge space $X_e$ has a neighborhood homeomorphic to $X_e\times (0, 1)$.

When the underlying graph is a tree $T$, then $X$ is called a \emph{tree of spaces}. In this case, in order to keep our terminology consistent with the results in \cite{LafontBook}], we will equip $X$ with a map $p\colon X\to T$ which maps each vertex space to the corresponding vertex of $T$ such that for any edge $e\in T$ and $t$ in the internal part $e^{\circ}$ of $e$, $X_e\cong p^{-1}(t)$ and $p^{-1}(e^{\circ})$ is homeomorphic to $X_e\times (0, 1)$. If $X$ is equipped with a Riemannian metric we say that an \emph{internal wall} of $X$ corresponding to $e\in T$ is the closure $p^{-1}(e^{\circ})$. Similarly, a \emph{chamber} $C\subset X$ is the pre-image of a vertex of $T$. Two distinct chambers are said to be \emph{adjacent} if the corresponding vertices in $T$ are joined by an edge. Finally, a wall $W$ is adjacent to the chamber $C$ if $C\cap W\neq \emptyset$.
\end{define}

\begin{remark}
An extended graph $n$-manifold as in Definition \ref{geometric} is a graph of spaces where the vertex spaces are the geometric pieces and the edges spaces are components of the boundary of these pieces. The universal cover $\widetilde{M}$ is a tree of spaces $(\widetilde{M}, p, T)$, where for any edge $e$ we have that $p^{-1}(e^{\circ})$ is diffeomorphic to $\IR^{n-1}\times (0, 1)$. For this it is crucial we restrict ourselves to affine diffeomorphism of the boundaries. The main issue here is that one could take connected sums of the edges with exotic spheres; see Remark 2.5. in \cite{LafontBook} for more details. 
\end{remark}

Next, we observe that  $\pi_{1}(\bar{N}_i)$ injects in the fundamental group of $\pi_1(M)$, where $\pi_{1}(M)$ is isomorphic to the graph of groups corresponding to the geometric decomposition of $M$, see \cite{Hempel} and Section 2.3 in \cite{LafontBook}. Let $\widetilde{N}_i$ be the universal cover of a piece $\bar{N}_i$. We clearly have
\[
\widetilde{N}_i=B_i\times \IR^{n-n_i}
\]
where $B_i$ is the complement in $\IH^{n_i}_{\IR}$ of an equivariant family of disjoint open \emph{horoballs}. Putting a rough metric on $M$ as in \cite{LafontBook} we then get the following result for the topological universal cover $\pi:\widetilde{M}\to M$.

\begin{corollary}\label{rough}
	$M$ admits a Riemannian metric such that $\widetilde{M}$ can be turned into a tree of spaces $(\widetilde{M}, \tilde{p}, \widetilde{T})$ such that:
	\begin{itemize}
		\item If $C$ is a chamber of $\widetilde{M}$, then $(C, \tilde{g}_C)$ is isometric to $B\times\IR^{k}$ where the Euclidean factor is equipped with a flat metric, and where $B$ is the complement of open horoballs in $(\IH^{n-k}_{\IR}, g_{\IR})$.
		\item If $W$ is an internal wall of $\widetilde{M}$, then $W$ is diffeomorphic to $\IR^{n-1}\times [-1, 1]$. 
	\end{itemize}
\end{corollary}

Corollary \ref{rough} tells us that the universal cover of any extended graph manifold is naturally divided into chambers equipped with locally symmetric metrics, either real-hyperbolic or modeled on the product of flat and real-hyperbolic geometries. These chambers are connected via walls diffeomorphic to $\IR^{n-1}\times [-1, 1]$. 

Similar results hold for pure complex-hyperbolic higher graph manifolds, and one can readily adapt the approach given in Chapter 2 of \cite{LafontBook}. The results we state below are somewhat implicit in \cite{CSS19}.

Fix an integer $n\geq 2$ for the underlying manifold $M^{2n}$, an integer $k\geq 1$ giving the numbers of complex-hyperbolic pieces of complex dimension $n$. For every $i$ let $W_i$ be a complete finite-volume complex-hyperbolic manifold of complex dimension $n$ with nilmanifold cusps. In other words, we have that $W_i=\Gamma_i \backslash \IH^{n}_{\IC}$ where $\Gamma_i\leq \textrm{PU}(n, 1)$ is a non-uniform torsion-free lattice whose parabolic elements have no rotational part, see Section 2 in \cite{DD17} for more details. Truncate the cusps of $W_i$ to obtain a compact manifold with nilmanifold boundaries, say $\bar{W}_i$. For any $i=1, ..., k$ we define the $2n$-dimensional manifold with boundary $M_i:=\bar{W}_i$. Notice that the nilmanifold boundaries of $M_i$ have a well-defined affine structure, see Sections 4.2.1 and 4.2.2. in \cite{Gol99}. Indeed, each horosphere in $\IH^{n}_{\IC}$ has a geometric structure modeled on the $(2n-1)$-Heisenberg group $\mathcal{H}^{2n-1}$ and the nilmanifold cross sections of the cusps are simply quotients of $\mathcal{H}^{2n-1}$ via lattices of translations. 

\begin{define}\label{geometric2}[\emph{cf}. Definition 1. in \cite{CSS19}]
	For $n\geq 2$ a pure complex-hyperbolic higher graph $2n$-manifold is a closed manifold built from $k\geq 1$ pieces $M_i$ as above by gluing their nilmanifold boundaries with affine diffeomorphisms.  
\end{define}

Similarly to the case of extended graph manifolds, we observe that  $\pi_1(W_i)=\pi_{1}(\bar{M}_i)$ injects in the fundamental group of $\pi_1(M)$, where $\pi_{1}(M)$ is isomorphic to the graph of groups corresponding to the geometric decomposition of $M$, see again \cite{Hempel} and Section 2.3 in \cite{LafontBook}. Let $\widetilde{M}_i$ be the universal cover of a piece $\bar{M}_i$. We clearly have
\[
\widetilde{M}_i=B_i
\]
where $B_i$ is the complement in $\IH^{n}_{\IC}$ of an equivariant family of disjoint open \emph{horoballs}. Putting a rough metric on $M$ as in Corollary \ref{rough}, we then get the following result for the topological universal cover $\pi:\widetilde{M}\to M$.

\begin{corollary}\label{roughc}
	$M$ admits a Riemannian metric such that $\widetilde{M}$ can be turned into a tree of spaces $(\widetilde{M}, \tilde{p}, \widetilde{T})$ such that:
	\begin{itemize}
		\item If $C$ is a chamber of $\widetilde{M}$, then $(C, \tilde{g}_C)$ is isometric to $B$, where $B$ is the complement of open horoballs in $(\IH^{n}_{\IC}, g_{\IC})$.
		\item If $W$ is an internal wall of $\widetilde{M}$, then $W$ is diffeomorphic to $\mathcal{H}^{2n-1}\times [-1, 1]$, where $\mathcal{H}^{2n-1}$ is the $(2n-1)$-dimensional Heisenberg Lie group. 
		\end{itemize}
\end{corollary}

\section{Dimension Estimates}\label{Pricepackage}

In this section, we repackage the Price inequality estimates for harmonic forms proved by the first author and M. Stern in \cite{DS17}, \cite{DS19}, and \cite{DS20}. The statements presented here are tailored to the study of higher graph manifolds, and they may slightly differ from the presentation in the original sources.

Let $(M^n, g)$ be a closed Riemannian manifold, and denote by $\mathcal{H}^k_{g}(M)$ the finite dimensional vector space of harmonic $k$-forms. By the Hodge-de Rham theorem, we know that
\[
b_k(M)=\dim_{\IR}\mathcal{H}^k_{g}(M)
\]
where $b_{k}(M)$ is the $k$-th topological Betti number of $M$. 

Let $K(\cdot,\cdot)$ denote the Schwartz kernel for the $L^2$ orthogonal projection onto $\mathcal{H}^{k}_{g}(M)$. Thus for  $x, y \in M$, we have  
\[
K(x, y) \in Hom( \Omega^{k}T^{*}_{y}M , \Omega^{k}T^{*}_{x}M).
\]
Given an $L^2$-orthonormal basis $\{\alpha_j\}_{j=1}^l$ for $\mathcal{H}^{k}_{g}(M)$, we have 
\begin{align}\label{KMap}
K(x, y) =\sum^{l}_{j=1}\alpha_{j}(x)\langle\cdot, \alpha_{j}(y)\rangle,
\end{align}
so that 
\begin{align}\label{Ktrace}
Tr K(x, x)=\sum^{\binom{n}{k}}_{i=1}\langle K(x, x)(e_{i}), e_{i}\rangle=\sum^{b_k(M)}_{i=1}|\alpha_{i}|^{2}_{x}\geq 0,
\end{align}
and 
\begin{align}\label{integral}
\int_{M}TrK(x, x)d\mu_g=b_{k}(M).
\end{align}

Because of Equation \eqref{integral}, we give the following definition.

\begin{definitions}\label{density}
	Given $(M, g)$, $\mathcal{H}^k_{g}(M)$, and an $L^2$-orthonormal basis $\{\alpha_j\}_{j=1}^{b_k(M)}$ for $\mathcal{H}^{k}_{g}(M)$, we define the pointwise $k$-th Betti number density to be
	\[
	\rho_{k}(x):=TrK(x, x),
	\]
	for any $x\in M$.
\end{definitions}

Next, we collect a linear algebra type pointwise estimate for the $k$-th Betti number density.

\begin{lemma}\label{keyestimate}
	Given  $(M, g)$ and $\mathcal{H}^k_{g}(M)$ as above, for any $x\in M$ we have
	\[
	0\leq \rho_{k}(x)\leq \binom{n}{k}\sup_{\alpha\in\mathcal{H}^{k}_{g}(M):||\alpha||^{2}_{L^{2}}=1}|\alpha(x)|^{2}.
	\] 
\end{lemma}
\begin{proof}
See Lemma 14 in \cite{DS19}.
\end{proof}

The following lemma is an application of the usual elliptic regularity and Moser iteration for harmonic forms in bounded geometry. 

\begin{lemma}\label{estimate}
	Given $(M^n, g)$ and $\mathcal{H}^k_{g}(M)$ as above, if the sectional curvature is bounded 
	\[
	|sec_g|\leq a
	\]
	for some $a\geq 0$, then there exists a constant $c=c(n, a, k, \inj_g(M, x))>0$ such that for any $x\in M$
	\[
	0\leq\rho_{k}(x)\leq c,
	\]
	where $\inj_g(M, x)$ is the injectivity radius at $x$.
\end{lemma}
\begin{proof}
	See Lemma 16 in \cite{DS19}. 
\end{proof}

Finally, we collect a couple of estimates for the pointwise Betti number density at points in closed Riemannian manifolds which possess geodesic balls centered around them isometric to large geodesic balls in the real- and complex-hyperbolic space. We start with the real-hyperbolic case.

\begin{lemma}\label{Restimate}
	Let $(M^n, g)$ and $\mathcal{H}^k_{g}(M)$ be as above. For any $R>1$ and $x\in M$ such that $B_g(x, R)$ is isometric to a ball $B_{g_{\IR}}(-, R)$ in the real-hyperbolic space form $(\IH^n_{\IR}, g_{\IR})$ (with sectional curvature normalized so that $sec_{g_{\IR}}=-1$), we have
	\begin{equation}\label{R1}
	0\leq\rho_k(x)\leq  \;     \begin{cases} 
	C_1(n, k)e^{-(n-1-2k)R} &\mbox{if }\quad    n-1-2k > 0 ,  \\
	C_{2}(n, k)R^{-1} & \mbox{if }\quad   n-1-2k = 0,
	\end{cases}
	\end{equation}
	where $C_{i}(n, k)>0$, $i=1, 2$, depend on $n$ and $k$ only.
\end{lemma}
\begin{proof}
	 The lemma is a reformulation, tailored to our present purposes, of the estimates for negatively $\epsilon$-pinched Riemannian manifolds originally proved in Theorems 87, 96, and Corollary 108 of \cite{DS17}. Set $\epsilon= 0$ in these estimates to extract the results needed for the proof of this lemma. 
\end{proof}

In the complex-hyperbolic case we have the following.

\begin{lemma}\label{Cestimate}
	Let $(M^{2n}, g)$ and $\mathcal{H}^k_{g}(M)$ be as above. For any $R>1$ and $x\in M$ such that $B_g(x, R)$ is isometric to a ball $B_{g_{\IC}}(-, R)$ in the complex-hyperbolic space form $(\IH^n_{\IC}, g_{\IC})$ (with sectional curvature normalized so that $-4\leq sec_{g_{\IC}}\leq-1$), we have
	\begin{equation}\label{C1}
	0\leq\rho_k(x)\leq  \;     \begin{cases} 
	D_1(n, k)e^{-2(n-1-k)R} &\mbox{if }\quad    n-1-k > 0 ,  \\
	D_{2}(n, k)R^{-1} & \mbox{if }\quad   n-1-k = 0,
	\end{cases}
	\end{equation}
	where $D_{i}(n, k)>0$, $i=1, 2$ depends on the $n$ and $k$ only.
\end{lemma}
\begin{proof}
	For this lemma, we need to go beyond the estimates proved in \cite{DS17} for negatively curved and pinched Riemannian manifolds. 
	Indeed, the estimate in \eqref{C1} crucially relies upon the fact that geodesic spheres in complex-hyperbolic space have constant mean curvature. For the details see Section 3 in \cite{DS20} and in particular Lemma 25 and Lemma 28 therein. 
\end{proof}

\section{Cheeger-Gromov-Fukaya Theory and Price Inequalities: Proofs of Theorems \ref{real} and \ref{complex}}

In the paper \cite{CSS19}, Connell-Su\'arez-Serrato compute exactly the minimal volume of extended graph manifolds (see Theorem 4 in \cite{CSS19}), and among other things are able to estimate from above and below the minimal volume of \emph{any} higher graph manifolds (see Theorem 3 in \cite{CSS19}). An important ingredient in their analysis is the construction of a particular sequence of metrics on any given higher graph manifold. This sequence of metrics is identically equal to the locally symmetric hyperbolic metrics on larger and larger portions of the pure hyperbolic pieces, and it collapses with bounded curvature the non-pure pieces and the gluing regions. The construction of such sequence of metrics relies on the general theory of collapsing developed by Cheeger-Gromov \cite{CG86} and a crucial subsequent refinement of Fukaya \cite{Fuk89}.
The following lemma is proven in Section 2, pages 1287 - 1292, of \cite{CSS19}. We prefer to collect as a lemma a particular case of the construction of Connell-Su\'arez-Serrato when specialized to extended graph manifolds (as defined in \cite{LafontBook}). For a very detailed and elementary proof of this lemma in the $3$-dimensional case, we also refer the interested reader to the Ph.D. thesis of Pieroni \cite{Pieroni}. 

\begin{lemma}[Connell-Su\'arez-Serrato]\label{cheeger}
	Let $M$ be an extended graph manifold with $k\geq 1$ real-hyperbolic pieces $\{(V_j, g_{\IR})\}^k_{j=1}$. There exists a sequence of metrics $\{g_m\}$ on $M$ such that:
	\begin{itemize}
		\item $\lim_{m\to\infty}Vol_{g_m}(M)=\sum^k_{j=1}Vol_{g_{\IR}}(V_j)$;
		\item $|sec_{g_m}|\leq 1+\delta_m$, where $\lim_{m\to\infty}\delta_m=0$;
		\item $g_m=g_{\IR}$ on each geometric piece $V_j$ up to height $m$ on each of its cusps. In particular, $\lim_{m\to\infty}diam_{g_m}(M)=\infty$.
	\end{itemize}
\end{lemma}
\begin{proof}
	Recall that, if $n=\dim M$, a cusp of a real-hyperbolic piece is isometric to $\IR^{+}\times T^{n-1}$ equipped with the metric $dt^2+e^{-2t}g_{T^{n-1}}$, where $g_{T^{n-1}}$ is a flat metric on a $(n-1)$-dimensional torus $T^{n-1}$. Up to height $m$ on a cusp then corresponds the region $[0, m]\times T^{n-1}$. For the remaining details of the construction of the sequence of metrics $g_{m}$, we refer to Section 2 in \cite{CSS19}.
\end{proof}

We can now prove the first theorem stated in the introduction.

\begin{proof}[Proof of Theorem \ref{real}]
	
Let $\Lambda:=\pi_{1}(M)$ be residually finite, and let  $\{\Lambda_i\}$ be a cofinal filtration of $\Lambda$ by nested finite index normal subgroups. For any fixed integer $m$, let $(M^i, g^i_m)$ be the associated tower of normal Riemannian coverings where
\[
\pi_i: M^i\to M, \quad g^i_{m}=\pi^*_i(g_m)
\]
are respectively the cover associated to the finite index normal subgroup $\Lambda_i \triangleleft\Lambda$ and the pull-back metric of $g_m$ via $\pi_i$. Here $(M, g_{m})$ is the sequence of metrics given in Lemma \ref{cheeger}.

For any integer $m$, we define $M_m$ to be the possibly disconnected open submanifold of $M$ given by the union up to height $m/2$ along the cusps of all the pure pieces $\{V_j\}^k_{j=1}$. Also, for any regular cover $\pi_i:M^i\to M$ we define
\[
M^i_m:=\pi^{-1}_m(M_m)\subset M^i,
\] 
where $M^i_m$ is an open and possibly disconnected submanifold of $M^i$.

\begin{lemma}\label{fat}
	For any integer $m\geq 1$, there exists a positive integer $i_m$ such that for $i\geq i_m$ and for $x\in M^i_m$ there exists a ball centered at $x$ of radius $m/2$, say $B_{g_i}(x, m/2)$, isometric to a ball in $(\IH^n_{\IR}, g_{\IR})$.
\end{lemma}
\begin{proof}
	First, for any $i, m\geq 1$ define
	\[
	r^i_m \quad \stackrel{{\rm def}}{=} \quad \inf \{d_{\tilde{g}_m}(z, \lambda_i z)\quad | \quad z\in \widetilde{M}, \quad \lambda_i\in\Lambda_i, \quad \lambda_i\neq \textrm{id}_{\Lambda_i} \}
	\]
	where $\tilde{\pi}:\widetilde{M}\to M$ is the topological universal cover and $\tilde{g}_m:=\tilde{\pi}^*g_m$. By Theorem 3.3 in \cite{DL19}, we know that for any fixed $m$
	\begin{align}\label{infinity}
	\lim_{i\to\infty}r^i_m=\infty.
	\end{align}
	Moreover, a Riemannian ball $B_{g^i_m}(-, r^i_m/2)$ (centered at any point) in $(M^i, g^i_m)$ is isometric to a ball of the same radius in $(\widetilde{M}, \tilde{g}_m)$. By construction, a ball $B_{g^i_m}(x, m/2)$ for $x\in M^i_m$ contains only points where the sectional curvature is equal to $-1$. Next, choose $i_m$ so that for $i\geq i_m$ we have $r^i_m\geq m/2$. This can always be achieved thanks to Equation \eqref{infinity}. We then have that for $i\geq i_m$, the ball $B_{g^i_m}(x, m/2)$ for $x\in M^i_{m}$ must be entirely contained in a real-hyperbolic chamber $(C, g_{\IR})$ of the tree of spaces $(\widetilde{M}, \tilde{\pi}, \tilde{T})$, see Corollary \ref{rough}.
\end{proof}

Using Equation \eqref{integral} and Definition \ref{density}, given any closed $n$-dimensional Riemannian manifold $(M, g)$ one can write
\[
b_l(M)=\int_{M}\rho_l(x)d\mu_g
\]
where $\rho_l(x)$ is the $l$-th Betti number density, see Section \ref{Pricepackage} for more details. Consider first the Betti numbers of sub-critical degree $l<(\dim_{\IR}(M)-1)/2$. We then have the estimate
\begin{align}\label{bettidensity}
b_{l}(M^i)=\int_{M^i_m}\rho_{l}(x)d\mu_{g^i_m}+\int_{M^i\setminus M^i_m}\rho_{l}(x)d\mu_{g^i_m}. 
\end{align}
By using Lemma \ref{fat}, we know that if $i>i_m$ then any point $x\in M^i_m$ has a ball $B_{g_i}(x, m/2)$ isometric to a ball in the real-hyperbolic space $(\IH^n_{\IR}, g_{\IR})$. Moreover, by eventually increasing $i_m$, we can also arrange that if $i>i_m$  then any point $x\in M^i\setminus M^i_m$ satisfies
\begin{align}\label{loweri}
inj_{g^i_m}(M^i, x)\geq \delta,
\end{align}
for some $\delta>0$. To see this, first recall that by Lemma 3.2. in \cite{LafontBook} any extended graph manifold is a $K(\pi, 1)$-space and then an \emph{essential} manifold in the terminology of Cheeger-Gromov \cite{CG86}. Now by construction of $M$ (as a graph), we have that $M$ admits an $F$-structure outside a compact subset, say $C$, compatible with the sequence of metrics $\{g_m\}$. The complement of $C$ in $M$ clearly contains all the non-pure pieces and the near infinity part of the pure real-hyperbolic pieces, i.e., the near infinity parts of the toral cusps. By Theorem 10 in \cite{CG86} the $F$-structure is injective, and as observed in Theorem 2.5 of \cite{Ron93} this easily implies that the sequence of metrics $\{g_m\}$ is collapsing with \emph{bounded covering geometry}. As the sequence $(M^i, g^i_m)$ is converging to the universal Riemannian cover, this is the same as saying that Equation \ref{loweri} holds (with $\delta$ independent of $m$) for $i$ large enough depending on $m$.

Next by Lemma \ref{Restimate}, if $x\in M^i_m$ and $i>i_m$ we have 
\begin{align}\label{rlemma}
0\leq\rho_l(x)\leq C_1(\dim_{\IR}(M), l) e^{-(\dim_{\IR}(M)-1-2l)\frac{m}{2}},
\end{align}
where $C_1(\dim_{\IR}(M), l)>0$ is a constant depending on the dimension and the degree only. Similarly, by Lemma \ref{estimate} if $x\in M^i\setminus M^i_m$ and $i>i_m$
\begin{align}\label{blemma}
0\leq\rho_l(x)\leq c(\dim_{\IR}(M), l, 1+\delta_m, \delta)
\end{align}
where $\delta$ is the injectivity radius lower bound coming from collapsing with bounded covering geometry given in Equation \eqref{loweri}, and where $1+\delta_m$ is the sectional curvature bound for the sequence of metrics $\{g_m\}$ given in Lemma \ref{cheeger}. As $\lim_{m\to\infty}\delta_m=0$ (see Lemma \ref{cheeger}), for $m>>0$ we can always assume that $0<1+\delta_{m}\leq 2$. Thus, for large $m$ the upper bound in Equation \eqref{blemma} can be chosen to be independent of $m$.

Next, observe that by construction of the metrics $\{g_{m}\}$ on $M$ we have
\begin{align}\label{volreminder}
\lim_{m\to\infty}Vol_{g_{m}}(M\setminus M_m)=0.
\end{align}
By combining Equations \eqref{bettidensity}, \eqref{rlemma}, \eqref{blemma}, and \eqref{volreminder} we have the estimate
\begin{align}\notag
b_{l}(M^{i})&\leq C_{1} Vol_{g^i_m}(M^i_m)e^{-(\dim_{\IR}(M)-1-2l)\frac{m}{2}}+cVol(M^i\setminus M^i_m)\\ \notag
&\leq C_1 Vol_{g^{i}_m}(M^i)e^{-(\dim_{\IR}(M)-1-2l)\frac{m}{2}}+c\deg(\pi_i)\epsilon(m),
\end{align}
where $\epsilon(m)>0$ is such that
\[
\lim_{m\to\infty}\epsilon(m)=0,
\]
see Equation \eqref{volreminder}. We then compute
\[
\limsup_{i\to\infty}\frac{b_l(M^i)}{Vol_{g^i_m}(M^i)}\leq C_1e^{-(\dim_{\IR}(M)-1-2l)\frac{m}{2}}+\frac{c\cdot\epsilon(m)}{Vol_{g_m}(M)}.
\]
Since the constant $c(n, l, 1+\delta_m, \delta)$ can be chosen independently of $m$ for $m>>0$, there exists $m_0>0$ such that for any $m\geq m_0$ we have
\begin{align}\label{limsupr}
\limsup_{i\to\infty}\frac{b_l(M^i)}{\deg(\pi_i)}\leq C_1Vol_{g_m}(M)e^{-(\dim_{\IR}(M)-1-2l)\frac{m}{2}}+c\cdot\epsilon(m),
\end{align}
where $c=c(\dim_{\IR}(M), l, \delta)>0$. By L\"uck's approximation theorem (see Main Theorem in \cite{Luck}), we know that
\begin{align}\label{luckr}
\limsup_{i\to\infty}\frac{b_l(M^i)}{\deg(\pi_i)}=\lim_{i\to\infty}\frac{b_l(M^i)}{\deg(\pi_i)}=b^{(2)}_l(M).
\end{align}
Thus, by taking the limit as $m\to\infty$ in the right-hand-side of Equation \eqref{limsupr} and combining it with Equation \eqref{luckr}, we obtain that $b^{(2)}_l(M)=0$ for $l<(\dim_{\IR}(M)-1)/2$, and by Poincar\'e duality for $l>(\dim_{\IR}(M)-1)/2$. Finally, if $\dim_{\IR}(M)=2n$ and $l=n$, we can use these vanishing results to compute
\[
b^{(2)}_{n}(M)=\lim_{i\to\infty}\frac{b_{n}(M^i)}{\deg(\pi_i)}=\frac{(-1)^n\chi(M^{i})}{\deg(\pi_i)}=(-1)^n\chi_{top}(M).
\]
Next, a Meyer-Vietoris argument yields
\[
\chi_{top}(M)=\sum^k_{j=1}\chi(V_{j}),
\]
and this concludes the proof in the even dimensional case.

In the odd dimensional case, everything is the same with the exception of the critical degree $l=(\dim_{\IR}(M)-1)/2$. In this degree, Lemma \ref{Restimate} implies that if $x\in M^i_m$ and $i>i_m$ we have
\begin{align}\label{rlemma2}
0\leq\rho_l(x)\leq C_2(\dim_{\IR}(M), l)\frac{2}{m},
\end{align}
\emph{cf}. Equation \eqref{rlemma}. The rest of the argument now goes through simply interchanging Equation \eqref{rlemma} with Equation \eqref{rlemma2}. As the upper bound in Equation \eqref{rlemma2} goes to zero when $m\to\infty$, we still obtain in the end that $b^{(2)}_l(M)=0$ and then by Poincar\'e duality $b^{(2)}_{l+1}(M)=0$ also. The proof is complete.
\end{proof}
 
We can now focus on the case of pure complex-hyperbolic higher graph manifolds. The following lemma provides a suitable sequence of metrics on such manifolds which collapses with bounded covering geometry the gluing regions in the graph manifold, and it is locally symmetric on larger and larger portions of the vertices. This sequence of metrics was explicitly constructed by Connell-Su\'arez-Serrato in \cite{CSS19}.

\begin{lemma}[Connell-Su\'arez-Serrato]\label{cheegerc}
	Let $M^{2n}$ be a pure complex-hyperbolic higher graph manifold with $k\geq 2$ pure complex-hyperbolic pieces, say $\{(W_j, g_{\IC})\}^k_{j=1}$. There exists a sequence of metrics $\{g_m\}$ on $M$ such that:
	\begin{itemize}
		\item $\lim_{m\to\infty}Vol_{g_m}(M)=\sum^k_{j=1}Vol_{g_{\IC}}(W_j)$;
		\item $|sec_{g_m}|\leq 4$ for any $m\geq 1$;
		\item $g_m=g_{\IC}$ on each piece $W_j$ up to height $m$ on each of its cusps. In particular, $\lim_{m\to\infty}diam_{g_m}(M)=\infty$.
	\end{itemize}
\end{lemma}
\begin{proof}
	See Section 2 in \cite{CSS19}, pages 1287-1293.
\end{proof}

We can now prove the second theorem stated in the introduction.

\begin{proof}[Proof of Theorem \ref{complex}]
	
	The set up is very similar to the proof of Theorem \ref{real}. Let $\Lambda:=\pi_{1}(M)$ be residually finite, and let  $\{\Lambda_i\}$ be a cofinal filtration of $\Lambda$ by nested finite index normal subgroups. For any fixed integer $m$, let $(M^i, g^i_m)$ be the associated tower of normal Riemannian coverings where
	\[
	\pi_i: M^i\to M, \quad g^i_{m}=\pi^*_i(g_m)
	\]
	are respectively the cover associated to the finite index normal subgroup $\Lambda_i \triangleleft\Lambda$ and the pull-back metric of $g_m$ via $\pi_i$. Here $(M, g_{m})$ is the sequence of metrics given in Lemma \ref{cheeger}.
	
	For any integer $m$, we define $M_m$ to be the possibly disconnected open submanifold of $M$ given by the union up to height $m/2$ along the cusps of all the pure pieces $\{V_j\}^k_{j=1}$. Also, for any regular cover $\pi_i:M^i\to M$ we define
	\[
	M^i_m:=\pi^{-1}_m(M_m)\subset M^i,
	\] 
	where $M^i_m$ is an open and possibly disconnected submanifold of $M^i$.
	
	\begin{lemma}\label{fatc}
		For any integer $m\geq 1$, there exists a positive integer $i_m$ such that for $i\geq i_m$ and for $x\in M^i_m$ there exists a ball centered at $x$ of radius $m/2$, say $B_{g_i}(x, m/2)$, isometric to a ball in $(\IH^n_{\IC}, g_{\IC})$.
	\end{lemma}
	\begin{proof}
		Use Corollary \ref{roughc} instead of Corollary \ref{rough} in the proof of Lemma \ref{fat}. The rest is identical.
	\end{proof}
	
	Given Lemma \ref{fat} and Lemma \ref{cheegerc}, we write for $l<n=\dim_{\IR}(M)/2$
	\[
	b_l(M)=\int_{M}\rho_l(x)d\mu_g
	\]
	where $\rho_l(x)$ is the $l$-th Betti number density, see Section \ref{Pricepackage} for more details. From this point on, one can then follow step-by-step the proof of Theorem \ref{real} in the even dimensional case. The only difference is the use of the Price inequality estimate for complex-hyperbolic spaces given Lemma \ref{Cestimate} instead of the estimate for real-hyperbolic spaces given in Lemma \ref{Restimate}. The proof is complete.
	\end{proof}

\section{Higher Graph Manifolds Without Pure Pieces}\label{Gromov}

Recall the definition of Gromov minimal volume.

\begin{define}[Gromov \cite{Gro82}]
	Let $M^n$ be a smooth manifold. Consider the set of all complete metrics on $M$ satisfying the bound
	\[
	|sec_g|\leq 1.
	\]
	We define the Gromov minimal volume of $M$ to be
	\[
	\text{MinVol(M)}:=\inf_{g} \Big\{ Vol_{g}(M)\text{  } | \text{  } |sec_{g}|\leq 1\Big\}.
	\]
\end{define}

As stated in \cite{Gro99}, for each $n$ there is a constant $c_n$ with the following property: If $M^n$ is an $n$-dimensional, closed, aspherical manifold, then
\begin{align}\label{Sauer}
b^{(2)}_i(M)\leq c_n MinVol(M),
\end{align}
for all $i\geq 0$. In \cite[Section 5.33]{Gro99} Gromov outlines a proof of this fact under the assumption that $\pi_1(M)$ is residually finite. See the paper by Sauer \cite{Sau09} for a proof of Equation \eqref{Sauer} without the residual finiteness assumption. In particular, the Singer conjecture holds true for aspherical manifolds with $MinVol=0$ (see Conjecture \ref{Singer}). Finally, we observe that if the fundamental group is residually finite then Fauser-Friedl-L\"oh prove an integral approximation theorem for the simplicial volume of higher graph manifolds without pure pieces, see Theorem 1.2 in \cite{Loh}. This interesting result can then be used to prove the vanishing of gradient invariants of higher graph manifolds without pure pieces, see Corollary 1.4 and Remark 1.5 in \cite{Loh}. Gradient invariants are defined in \cite{Loh} as the normalization over finite coverings and finite index subgroups of several topological quantities. In the case of Betti numbers, this normalization process recovers the usual $L^2$-Betti numbers. For more details, see the Introduction in \cite{Loh}.

We now provide a proof of the Singer conjecture (and Equation \eqref{Sauer}) for aspherical manifold with $MinVol=0$ and residually finite fundamental group.  The proof is along the lines of the strategy employed in the proofs of Theorem \ref{real} and Theorem \ref{complex}. Even if this was already known thanks to the work of Gromov and Sauer, we think it is somewhat interesting to have yet another approach.

\begin{proposition}\label{collapsing}
	Let $M^n$ be a closed aspherical manifold with $MinVol(M)=0$ and $\pi_1(M)$ residually finite. We then have
	\[
	b^{(2)}_{l}(M)=0, \quad \textrm{for\quad $l=0, 1, ..., n$}.
	\]
\end{proposition}
\begin{proof}
	Since $MinVol(M)=0$, there exists a sequence of metrics $\{g_m\}$ on $M$, with $|sec_{g_m}|\leq 1$, and such that
	\[
	\lim_{m\to\infty}Vol_{g_m}(M)=0.
	\]
	Given $\Lambda:=\pi_1(M)$, let a $\{\Lambda_i\}$ be a cofinal filtration by nested finite index normal subgroups. For any fixed $m$, let $(M^i, g^i_m)$ be the associated tower of normal Riemannian coverings where
	\[
	\pi_i: M^i\to M, \quad g^i_{m}=\pi^*_i(g_m)
	\]
	are respectively the cover associated to $\Lambda_i \triangleleft\pi_1(M)$ and the pull-back metric of $g_m$ via $\pi_i$. Next, we have the following lemma.
	\begin{lemma}\label{cofinal}
		For any fixed $m$, there exists $i_m>0$ such that for any $i\geq i_m$ we have
		\begin{align}\label{lowi}
		\inj_{g^i_m}(M^i)\geq \epsilon>0,
		\end{align}
		where $\epsilon>0$ is independent of $m$.
	\end{lemma}
	\begin{proof}
		$M$ is an essential manifold in the terminology of Cheeger-Gromov \cite{CG86}, and we have that $M$ admits an $F$-structure. By Theorem A.1. in \cite{CG86} the $F$-structure is injective, and as observed in Theorem 2.5 of \cite{Ron93} this easily implies that the sequence of metrics $\{g_m\}$ is collapsing with \emph{bounded covering geometry}. As the sequence $(M^i, g^i_m)$ is converging to the universal Riemannian cover, this gives that Equation \eqref{lowi} is satisfied. 
	\end{proof}

We can now use Lemma \ref{cofinal} with Lemma \ref{estimate}, to conclude that
\begin{align}\label{price}
\frac{b_l(M^i)}{Vol_{g^i_k}(M^i)}\leq C,
\end{align}
for $i\geq i_m$, with $C>0$ independent of $m$. We can now apply L\"uck's approximation theorem (see Main Theorem in \cite{Luck}) to conclude that for any fixed $m$
\begin{align}\label{luck}
\lim_{i\to\infty}\frac{b_{l}(M^i)}{Vol_{g^i_m}(M^i)}=\lim_{i\to\infty}\frac{b_{l}(M^i)}{\deg(\pi_i)Vol_{g_m}(M)}=\frac{b^{(2)}_l(M)}{Vol_{g_m}(M)}.
\end{align}
By combining Equations \eqref{price} and \eqref{luck}, we then obtain that
\[
b^{(2)}_l(M)\leq C\cdot Vol_{g_{m}}(M),
\]
for any $m$ (\emph{cf}. Equation \eqref{Sauer}). The proof is therefore complete.
\end{proof}

The next corollary is then clear.

\begin{corollary}
	Let $M^n$ be a higher graph manifold without pure pieces and with $\pi_1(M)$ residually finite. We then have
	\[
	b^{(2)}_l(M)=0
	\]
	for all $l\geq 0$.
\end{corollary}

\begin{proof}
	By Theorem 1 in \cite{CSS19}, we know that $MinVol(M)=0$. By Lemma 3.2. in \cite{LafontBook}, we know that $M$ is aspherical. Conclude the proof by using Proposition \ref{collapsing}.
\end{proof}

\section{Residual Finiteness of Higher Graph Manifolds}


In this section, we study the residual finiteness of fundamental groups of higher graph manifolds. In general, it is an open question whether all such manifolds have residually finite fundamental group, see Section 12.1 in \cite{LafontBook}. We will first show that a large class of pure real-hyperbolic extended graph manifolds have residually finite fundamental group using Wise's virtually special machinery (Theorem \ref{vspecial}). This class includes, for example, any such manifold where the pieces are all arithmetic (Corollary \ref{arithmeticrf}). However, this machinery does not apply in the complex-hyperbolic setting. Hence, we will study applications of Hempel's approach to the residual finiteness of 3--manifold groups in higher dimensions. A conjecture of Behrstock-Neumann implies that this approach works quite broadly, see Proposition \ref{ccc->rf}. However, even without this conjecture we are able to apply this approach to a procedure we call \emph{iterated (twisted) partial doubles} to construct infinite families of both pure real-hyperbolic extended graph manifolds and pure complex-hyperbolic higher graph manifolds. Finally, we give some explicit examples of manifolds which can be used to construct these families. The pure complex-hyperbolic higher graph manifolds we construct appear to be the \emph{first} explicit examples of such manifolds available in the literature.

\subsection{Wise's Virtually Special Machinery and Pure Arithmetic Real-Hyperbolic Higher Graph Manifolds}


First, we observe that Wise's virtually special machinery implies residual finiteness of the fundamental groups of a large class of pure real-hyperbolic extended graph manifolds. This general result applies in particular to the case where all pieces are real-hyperbolic and arithmetic, though we expect it to apply more broadly as well. In fact, we are not aware of any real-hyperbolic manifolds which are known to not satisfy the hypothesis of Theorem \ref{vspecial}.

A cube complex is a space built by gluing cubes of various dimensions together along faces, where a cube is a copy of $[0, 1]^n$ for some $n$. A cube complex $C$ is called \emph{special} if $C$ admits a local isometry into the Salvetti complex of a right-angled Artin group, and $C$ is \emph{virtually special} if it has a finite cover which is special. A group $G$ is called virtually special if $G\cong \pi_1(C)$ where $C$ is a virtually special cube complex. This implies that $G$ has a finite index subgroup which embeds into a right-angled Artin group. $G$ is \emph{virtually compact special} if, in addition, $C$ is compact and \emph{virtually sparse special} if $C$ is the union of a compact subset and certain cusp-like subspaces. We refer to \cite{Wis21} for precise definitions and properties of virtually special groups. 

The next result is a special case of Theorem 1.9 in \cite{HuaW}. To see that the hypotheses are satisfied, we note that by Dahmani's combination theorem for relatively hyperbolic groups \cite{Dah03} $\pi_1(M)$ will be hyperbolic relative to the collection of maximal parabolic subgroups of the fundamental group of each piece. These parabolic subgroups will be relatively quasi-convex by definition \cite{Hru10}. 

\begin{theorem}\label{vspecial}
	Let $M$ be a pure real-hyperbolic extended graph manifold such that the fundamental group of each piece is virtually sparse special. Then $\pi_1(M)$ is residually finite. 
\end{theorem}

Arithmetic lattices were shown to be virtually special in \cite{BHW11}. While they focused on uniform lattices, the construction produces a virtually special but not compact cube complex for non-uniform lattices as well, see Section 9 in \cite{BHW11} and Section 4 in \cite{BW12}. Finiteness properties of this method of producing cube complexes were studied in \cite{HruW14}, where they were shown to be relatively compact in the non-uniform case, see Corollary 7.15 in \cite{HruW14}. They also show that the additional condition of the stronger property of being sparse also holds provided one makes an appropriate choice of cubulation of the peripheral structure, see Theorem 7.12 in \cite{HruW14}.

\begin{corollary}\label{arithmeticrf}
Let $M$ be a pure real-hyperbolic extended graph manifold such that each piece is a real-hyperbolic arithmetic manifold. Then $\pi_1(M)$ is residually finite. 
\end{corollary}

Corollary \ref{arithmeticrf} together with Theorem \ref{real} implies Theorem \ref{arithmetic}.


\subsection{Hempel's Criterion and the Cusp Covering Conjecture}

The virtually special machinery cannot be applied in the complex-hyperbolic setting, see Remark 4.4 in \cite{BW12}. In order to produce examples of pure complex-hyperbolic higher graph manifolds with residually finite fundamental group, we extend Hempel's approach to residual finiteness for 3-manifold groups to certain classes of higher graph manifolds.  

Hempel's proof has two main ingredients, separability and compatibility. Compatibility is a condition which allows a collection of covers of the geometric pieces to be assembled into a cover of the whole manifold. When combined with separability of the boundary subgroups, any loop in the manifold can be lifted to a path in the finite cover assembled in this way, which shows that the fundamental group is residually finite. We now describe these conditions in more detail.

Let $X$ be a graph of spaces (see definition \ref{graphofspaces}), and let $V_1,..., V_m$ denote the vertex spaces of $X$. For each $V_i$,  Let $\overline{V}_i$ denote a covering space of $V_i$ with covering map $p_i\colon\overline{V}_i\to V_i$. The collection of covers $\overline{V}_1,...,\overline{V}_m$ is called \emph{compatible} if whenever $V_i$ and $V_j$ are adjacent vertex spaces glued along an edge space $E$,  then, for each component $\overline{E}$ of $p_i^{-1}(E)$ and each component $\overline{E}'$ of $p_j^{-1}(E)$, the induced covers $\overline{E}\to E$ and $\overline{E}'\to E$ are equivalent. Note that this implies in particular that for any two components of $ p_i^{-1}(E)$ the covering maps to $E$ must be equivalent.

Any compatible collection of finite covers can be assembled into a finite cover of $X$, see Theorem 2.2 in \cite{Hempel}.

\begin{theorem}\label{compcovers}
Let $X$ be a graph of spaces and let $V_1,..., V_m$ denote the vertex spaces of $X$. Let $\overline{V}_1,...,\overline{V}_m$ be a compatible collection of finite covers. Then $X$ has a finite cover $\overline{X}$ which is a graph of spaces where each vertex space is a copy of some $\overline{V}_i$.
\end{theorem}

By the Seifert-Van Kampen Theorem $\pi_1(X)$ splits as graph of groups with vertex groups given by fundamental groups of the vertex spaces and edge groups given by fundamental groups of the edge spaces, see Section 2.3 in \cite{LafontBook}. Dual to this graph of groups decomposition of $\pi_1(X)$ is an action of $\pi_1(X)$ on a tree, called the Bass-Serre tree of the splitting. In this case the Bass-Serre tree can be identified with the tree $T$ coming from the tree of spaces structure on the universal cover $\widetilde{X}$, with the action of $\pi_1(X)$ on $T$ induced by the action of $\pi_1(X)$ on $\widetilde{X}$. Each vertex of $T$ has a stabilizer which is conjugate to a vertex group in the graph of groups and each adjacent edge stabilizer is conjugate to an adjacent edge group.

Next, we give an algebraic version of the above compatibility condition in terms of this graph of groups structure. Suppose $G$ is a group which splits as a graph of groups with vertex set $V$ and edge set $E$ and let $\{G_v\}_{v\in V}$ and $\{G_e\}_{e\in E}$ denote the corresponding vertex and edge groups in this graph of groups. The graph of groups structure comes with injective homomorphisms $\phi_{e, v}\colon G_e\to G_v$ whenever $v$ is an endpoint of the edge $e$. However, we will identify $G_e$ with the subgroup $\phi_{e, v}(G_e)$ of $G_v$ and drop the maps $\phi_{e, v}$ from the notation. Now collections of finite-index normal subgroups $\{N_v\leq G_v\}_{v\in V}$ and $\{N_e\leq G_e\}_{e\in E}$ are called \emph{compatible} if whenever $v$ is an endpoint of the edge $e$, then $N_v\cap G_e=N_e$. In the case where $G=\pi_1(X)$ for a graph of spaces $X$ and the graph of groups is the one described above,  then given a compatible collection of finite-index normal subgroups, the covers induced by these subgroups will be a compatible collection of covers. 

The second ingredient in Hempel's approach is the \emph{separability} of the boundary subgroups of each geometric piece. This algebraic condition ensures that a loop based at a point on the boundary can be lifted in a finite cover to a path whose initial and terminal points are in different components of the boundary. 

\begin{define}
A subgroup $H\leq G$ is separable in $G$ if for all $g\in G\setminus H$, $G$ has a finite-index normal subgroup $N$ with $g\notin HN$.
\end{define}

We can now state Hempel's criterion for the residual finiteness of a fundamental group of a graph of groups, which is Theorem 3.1  from \cite{Hempel}.

\begin{theorem}\label{HempCrit}
Let $G$ be the fundamental group of a graph of groups with vertex set $V$ and edge set $E$ such that for each vertex group $G_u$ and each $g\in G_u\setminus\{1\}$, there exist  collections of finite index normal subgroups $\{N_v\leq G_v\}_{v\in V}$ and $\{N_e\leq G_e\}_{e\in E}$ which satisfy the following:
\begin{enumerate}
\item The collections $\{N_v\}_{v\in V}$ and $\{N_e\}_{e\in E}$ are compatible.
\item $g\notin N_u$.
\item If $u$ is an endpoint of the edge $e$ and $g\notin G_e$, then $g\notin G_eN_u$. 
\end{enumerate}

Then $G$ is residually finite.
\end{theorem}
Hempel showed Theorem \ref{HempCrit} can be applied to any 3-dimensional extended graph manifold \cite{Hempel}, which by the Geometrization Theorem implies that all 3-manifolds have residually finite fundamental group (see Section \ref{revisited} for more details).

Our aim is to apply this theorem to the fundamental groups of higher graph manifolds. A necessary condition for it to apply is that the boundary subgroups of each geometric piece are separable. We show this is indeed the case and moreover that boundary subgroups of any higher graph manifold $M$ are separable in $\pi_1(M)$. Our main tool for this is the following lemma, which is a special case of a well-known algebraic result, see for example the final Proposition in \cite{Long}.
 \begin{lemma}\label{maxnil}
 Let $G$ be residually finite and let $H$ be a maximal abelian or maximal nilpotent subgroup. Then $H$ is separable in $G$.
 \end{lemma}
\begin{proof}
Suppose that $g\in G\setminus H$. Define $[x_1, x_2]=x_1^{-1}x_2^{-1}x_1x_2$ and inductively define $[x_1,...,x_i]=[x_1, [x_2,...,x_i]]$ for all $x_1,...,x_i\in G$ and $i\geq 2$. That is, $[x_1,...,x_i]$ is the $i$-fold commutator. Since $H$ is nilpotent, for some $k\geq 2$ and for all $x_1,...,x_{k}\in H$, $[x_1,...,x_k]=1$. Moreover, since $H$ is maximal nilpotent, there exists $g_1,..., g_k\in\langle H, g\rangle$ with $[g_1,...,g_k]\neq 1$. Since $G$ is residually finite, there exists a finite index normal subgroup $N$ with $[g_1,...,g_k]\notin N$. But the image of $H$ in $G/N$ is still nilpotent of the same degree, hence it cannot contain the image of $g$. Thus $g\notin HN$. 
\end{proof}


\begin{lemma}\label{boundarysep}
Let $M$ be an extended graph manifold or a pure complex-hyperbolic higher graph manifold with boundary. Then for each boundary component $B$, $\pi_1(B)$ is separable in $\pi_1(M)$. 
\end{lemma}

\begin{proof}
We first consider the case where $M$ is an extended graph manifold. In this case, $B$ is the boundary component of some piece $N\cong V\times T^m$, where $V$ is a compact manifold with toral boundary whose interior admits a finite volume real-hyperbolic metric, and $B=C\times T^m$, where $C$ is a toral boundary component of $V$. Hence $B$ is a torus and $\pi_1(B)\cong \pi_1(C)\times \pi_1(T^m)$, is abelian.

Let $v$ be the vertex in the Bass-Serre tree stabilized by $\pi_1(N)$. If $e$ is an edge adjacent to $v$, then its stabilizer is conjugate to $\pi_1(C')\times \pi_1(T^m)$ where $C'\times T^m$ is the boundary component corresponding to the edge $e$. $\pi_1(C)$ has trivial intersection with any conjugate of $\pi_1(C')$, since these are distinct maximal parabolic subgroups inside $\pi_1(V)$. Hence if $g$ is a non-trivial element of $\pi_1(C)$, then $v$ is the unique vertex fixed by $g$. Now if $h$ is any element of $\pi_1(M)$ which commutes with $g$, then $ghv=hgv=hv$, and since the fixed point of $g$ is unique, we have $hv=v$. In particular, this implies that $h\in \pi_1(N)$. Now, since $\pi_1(C)$ is maximal parabolic subgroup of $\pi_1(V)$, it is also maximal abelian subgroup in $\pi_1(V)$ and hence we have that $\pi_1(C)\times\pi_1(T^m)\cong\pi_1(B)$ is a maximal abelian subgroup of $\pi_1(N)$. Hence $\pi_1(B)$ is also maximal abelian in $\pi_1(M)$ and is thus separable in $\pi_1(M)$ by Lemma \ref{maxnil}.

Now suppose that $M$ is a pure complex-hyperbolic higher graph manifold. In this case, $B$ is the boundary component of a piece $W$ which is a compact manifold with nilmanifold boundary whose interior admits a finite volume complex-hyperbolic metric. As in the previous case $v$ is the unique vertex fixed by any non-trivial element of $\pi_1(B)$. Let $K$ be a nilpotent subgroup of $\pi_1(M)$ which contains $\pi_1(B)$. Let $z$ be a non-trivial element of the center of $K$. 

Since $z$ commutes which each element of $\pi_1(B)$, $z$ also fixes $v$, hence $z$ is contained in $\pi_1(W)$.  Since $\pi_1(B)$ is a maximal parabolic subgroup of $\pi_1(W)$, it is a maximal nilpotent subgroup of $\pi_1(W)$, so $z\in \pi_1(B)$. It follows that $v$ is the unique fixed point of $z$ and since each element of $K$ commutes with $z$, we have that $K$ stabilizes $v$ and so $K\leq \pi_1(W)$. Since $\pi_1(B)$ is a maximal nilpotent subgroup of $\pi_1(W)$, $K=\pi_1(B)$ and so $\pi_1(B)$ is also maximal nilpotent in $\pi_1(M)$. Again, this implies that $\pi_1(B)$ is separable by Lemma \ref{maxnil}.
\end{proof}

While separability can be generalized to higher dimensions as above, Hempel's proof that compatible normal subgroups can be chosen is quite special to the 3-dimensional setting. The problem of funding such subgroups in higher dimensions is closely related to the cusp covering conjecture of Behrstock-Neumann \cite{BN12}. We give a slightly stronger version of that conjecture since  Behrstock-Neumann do not ask for the covers to be regular. 

\begin{conjecture}[Cusp covering conjecture]\cite{BN12}
Let $M$ be a hyperbolic n-manifold. Then for each cusp $C$ of M there exists a sublattice $\Lambda_C$ of $\pi_1(C)$ such that, for any choice of a sublattice $\Lambda'_C
\subseteq \Lambda_C$ for each $C$, there exists a finite regular cover $M'$ of $M$
whose cusps covering each cusp $C$ of $M$ are the covers determined by $\Lambda'_C$.
\end{conjecture}

\begin{proposition}\label{ccc->rf}
Let $M$ be a pure real-hyperbolic extended graph manifold. Then the cusp covering conjecture implies that $\pi_1(M)$ is residually finite. 
\end{proposition}

\begin{proof}


Consider the splitting of $\pi_1(M)$ as a graph of groups. Let $G_e$ be an edge group with adjacent vertex groups $G_v$ and $G_u$. The cusp covering conjecture provides $G_e$ with two finite index subgroups $\Lambda_1$ and $\Lambda_2$ where $\Lambda_1$ comes from considering $G_e$ as a cusp subgroup of $G_v$ and $\Lambda_2$ comes from considering $G_e$ as a cusp subgroup of $G_u$. Define $\Lambda_e=\Lambda_1\cap \Lambda_2$.

Now let $G_u$ be a vertex group and let $g\in G_u\setminus \{1\}$. Since $G_u$ is residually finite, there exists a finite index normal subgroup $K\leq G_u$ with $g\notin K$. For each adjacent edge group $G_e$ embedded in $G_u$, there is a finite index normal subgroup $K_e$ with $g\notin G_eK_e$ whenever $g\notin G_e$ (if $g\in G_e$ then we define $K_e=G_u$). Define $K_u$ to be the intersection of $K$ and each $K_e$. Then $K_u$ is a finite index normal subgroup of $G$. 

Now for each adjacent edge group $G_e$, the cusp covering conjecture says that $G_u$ contains a finite index normal subgroup $H$ with $H\cap G_e=\Lambda_e$. Finally, let $N_u=H\cap K_u$. Now $N_u$ is a finite index normal subgroup of $G$ such that $g\notin N_u$, $g\notin G_eN_u$ whenever $g\notin G_e$, and $N_u\cap G_e\leq \Lambda_e$ for all adjacent edges $e$. For each of these edges, let $N_e=N_u\cap G_e$.

Now for each edge group $G_e$ which is not adjacent to $G_v$, we define $N_e=\Lambda_e$. For each other vertex group $G_v$, by the cusp covering conjecture there exists a finite index normal subgroup $N_v\leq G_v$ such tha for each edge group $G_e$ adjacent to $G_v$, $N_v\cap G_e=N_e$. Hence the collections $\{G_v\}_{v\in V}$ and $\{G_e\}_{e\in E}$ satisfy the hypothesis of Theorem \ref{HempCrit}, so $\pi_1(M)$ is residually finite.


\end{proof}

While Behrstock-Neumann only considered real-hyperbolic manifolds, there is a natural analogue of this conjecture for complex-hyperbolic manifolds. The proof of Proposition \ref{ccc->rf} can be adapted in the obvious way to show that this would imply that pure complex-hyperbolic higher graph manifolds have residually finite fundamental group.


\begin{proposition}
Let $M$ be a pure complex-hyperbolic higher graph manifold. Then the complex-hyperbolic analogue of the cusp covering conjecture implies that $\pi_1(M)$ is residually finite. 
\end{proposition}

\subsection{Iterated Partial Doubles}

Without the cusp covering conjecture we are still able to construct infinite families of higher graph manifolds with residually finite fundamental group. This construction is based on the following algebraic lemma.

Let $G$ be a group and $\mathcal{H}=\{H_1,...,H_m\}$ a collection of subgroups of $G$. Let $G\ast_{\mathcal{H}} G$ denote the fundamental group of the graph of groups with two vertices  both labeled by $G$ and one edge labeled by each $H_i$.

\begin{lemma}\label{algebraic double}
Let $G$ be a residually finite group and $\mathcal{H}$ a finite collection of separable subgroups of $G$. Then $G\ast_\mathcal H G$ is residually finite.
\end{lemma}
\begin{proof}
We will show that $G$ is residually finite using Theorem \ref{HempCrit}. To that end, let $g$ be a non-trivial element of $G$. Since $G$ is residually finite, there exists a finite index normal subgroup $N_0$ with $g\notin N_0$. Let $\mathcal H=\{H_1,..., H_m\}$; by separability for each $1\leq i\leq m$ there exists a finite index normal subgroup $N_i$ with $g\notin H_iN_i$ whenever $g\notin H_i$ (if $g\in H_i$ then we define $N_i=G$). Define
\[
N=\bigcap_{i=0}^m N_i
\]
and note that $N$ is a finite index normal subgroup of $G$. We now choose a copy of $N$ inside each vertex group of  $G\ast_\mathcal H G$ and choose $N\cap H_i$ in the edge group labeled by $H_i$. This is a compatible collection of subgroups for this splitting. Since $g\notin N_0$ and $N\subseteq N_0$, we have $g\notin N$. Similarly, whenever $g\notin H_i$ we have $g\notin H_iN_i$ and hence $g\notin H_iN$. Thus, $N$ satisfies the hypothesis of Theorem \ref{HempCrit}, so $G\ast_\mathcal H G$ is residually finite.

\end{proof}

Let $M$ be a higher graph manifold and let $\mathcal B=\{B_1,...,B_m\}$ be a set of components of the boundary $\partial M$. We call the manifold obtained by taking two copies of $M$ and gluing each $B_i$ in one copy to the corresponding $B_i$ in the other copy by the identity map a \emph{partial double} of $M$. We denote this partial double by $M\cup_{\mathcal B} M$. If we let $G=\pi_1(M)$, and $H_i=\pi_1(B_i)$ and $\mathcal H=\{H_1,...,H_m\}$, then $\pi_1(M\cup_{\mathcal B} M)\cong G\ast_{\mathcal H} G$. So, combining Lemma \ref{boundarysep} and Lemma \ref{algebraic double} we obtain the following.

\begin{theorem}\label{doublerf}
Let $M$ be an extended graph manifold or a pure complex-hyperbolic higher graph manifold with boundary and suppose $\pi_1(M)$ is residually finite. Let $N$ be a partial double of $M$. Then $\pi_1(N)$ is residually finite.
\end{theorem}

Theorem \ref{doublerf} can be applied inductively to produce infinite families of higher graph manifolds with residually finite fundamental group. Explicitly, let $V$ be a fixed finite volume real-hyperbolic manifold with at least two cusps. Let $\bar{V}$ be the manifold with boundary obtained by truncating the cusps of $V$. Let $\mathcal B_0$ be a \emph{proper} subset of the components of $\partial \bar{V}$ and define $M_1=\bar{V}\cup_{\mathcal{B}_0}\bar{V}$. $M_1$ is now an extended graph manifold with boundary with $\pi_1(M_1)$ residually finite by Theorem \ref{doublerf}. Let $\mathcal B_1$ be a proper subset of the components of $\partial M_1$ and let $M_2=M_1\cup_{\mathcal{B}_1}M_1$ which again is an extended graph manifold with at least two boundary components and with $\pi_1(M_2)$ residually finite. Continuing inductively we produce a sequence $M_1$, $M_2$,... of extended graph manifolds with at least two boundary components and residually finite fundamental group. We call these manifolds \emph{iterated partial doubles} of $V$. For each $M_i$, let $M_i'$ denote the double of $M_i$ along the entire boundary of $M_i$. Then $M_i'$ will be a closed extended graph manifold with residually finite fundamental group. Moreover, $M_i'$ has $2^{i+1}$ geometric pieces which are all copies of $V$, so $\chi(M_i')=2^{i+1}\chi(V)$. In particular, the manifolds $M_1'$, $M_2'$,... are all distinct. Thus, for each finite volume real-hyperbolic manifold with at least two cusps, we have an infinite family  $M_1'$, $M_2'$ of distinct extended graph manifolds which satisfy the hypothesis of Theorem \ref{real}.


Starting with a finite volume complex-hyperbolic manifold $W$ with at least two cusps, we can construct a sequence of iterated partial doubles $M_1$, $M_2$,... of $W$ in the same way. Doubling each $M_i$ along the entire boundary to produce closed manifolds $M_i'$ as before, we obtain an infinite family  $M_1'$, $M_2'$ of distinct pure complex-hyperbolic graph manifolds which satisfy the hypothesis of Theorem \ref{complex}.

Given an extended graph manifold or a pure complex-hyperbolic graph manifold with boundary $M$, another way to produce a closed manifold is to start gluing components of the boundary of $M$ together. Algebraically, this produces a manifold whose fundamental group is an HNN extension of $\pi_1(M)$. Preserving residual finiteness under such constructions is quite tricky but can be done in some special cases, see Chapter 15 in \cite{Hempbook} where such constructions are studied in the context of 3-manifolds and \cite{BT78} where they are studied from an algebraic perspective. One case where this construction will produce a manifold with residually finite fundamental group is the following.

Let $G$ be a group with two collections of subgroups $\mathcal H=\{H_1,..., H_m\}$ and $\mathcal K=\{K_1,...,K_m\}$. Suppose $\alpha$ is a automorphism $\alpha\colon G\to G$ such that $\alpha(H_i)=K_i$ for each $i$. Consider the graph of groups with a single vertex with vertex group $G$ and with edges $e_1,..., e_m$ where the edge group for $e_i$ is $H_i$ attached via the identity map on one side and via the map $\alpha|_{H_i}\colon H_i\to K_i$ on the other. We denote the fundamental group of this graph of groups by $G\ast_{\mathcal H, \alpha}$.

The following is an extended version of Lemma 4.4 in \cite{BT78}.

\begin{lemma}\label{HNN}
Let $G$ be a residually finite group with two collections of separable subgroups $\mathcal H=\{H_1,..., H_m\}$ and $\mathcal K=\{K_1,...,K_m\}$. Suppose $\alpha$ is a automorphism $\alpha\colon G\to G$ such that $\alpha(H_i)=K_i$ for each $i$. Then $G\ast_{\mathcal H, \alpha}$ is residually finite.
\end{lemma}


\begin{proof}
Let $g\in G\setminus \{1\}$. Since $G$ is residually finite, there exists a finite index normal subgroup $N_0$ with $g\notin N_0$. Since each $H_i$ is separable, for each $i$ we can find a finite index normal subgroup $N_i$ such that if $g\notin H_i$ then $g\notin H_iN_i$ (whenever $g\in H_i$ we take $N_i=G$). Similarly choose finite index normal subgroups $L_i$ such that if $g\notin K_i$ then $g\notin K_iL_i$. Now let $N'=N_0\cap N_1\cap...\cap N_m\cap L_1\cap... L_m$, which is a finite index normal subgroup of $G$. Let $N$ be the intersection of all subgroups of $G$ with index $[G:N']$, and note that $N$ is finite index, normal, and $\alpha(N)=N$ since $\alpha$ maps any subgroup to one of the same index.

Let $N_{e_i}=N\cap H_i$, and note that for each $i$ the image of $N_{e_i}$ in $K_i$ is $\alpha(N_{e_i})$. We claim that $\{N\}$ and $\{N_{e_1},...,N_{e_m}\}$ are compatible collections of subgroups for this splitting of $G$. Clearly $N\cap H= N_e$, and also $N\cap K=\alpha(N)\cap \alpha(H)=\alpha(N\cap H)=\alpha(N_e)$. Since $N$ is contained in $N_0$, each $N_i$, and each $L_i$, the collections $\{N\}$ and $\{N_2\}$ will also satisfy conditions (2) and (3) of Theorem \ref{HempCrit}. Hence the $G\ast_{\mathcal H, \alpha}$ is residually finite by Theorem \ref{HempCrit}.

\end{proof}

An application of Lemma \ref{HNN} together with Lemma \ref{boundarysep} gives the following.

\begin{proposition}\label{handles}
Let $M$ be an extended graph manifold or a pure complex-hyperbolic higher graph manifold with boundary and suppose $\pi_1(M)$ is residually finite. Let  $\mathcal B=\{B_1,...,B_m\}$ and $\mathcal{B}'=\{B'_1,...,B'_m\}$ be two collections of boundary components of $M$ and suppose $f\colon M\to M$ is a homeomorphism with $f(B_i)=B'_i$ for each $i$. Then the manifold obtained by gluing each $B_i$ to $B_i'$ via the restriction of $f$ has residually finite fundamental group.
\end{proposition}


Let $M$ be a manifold and let $\mathcal B=\{B_1,...,B_m\}$ and $\mathcal{B}'=\{B'_1,...,B'_m\}$  be two collections of boundary components of $M$. Let $f\colon M\to M$ be a homeomorphism which maps each $B_i$ to $B_i'$. Let $M\cup_{\mathcal B, f} M$ denote the manifold obtained by taking two copies of $M$ and gluing each $B_i$ in the first copy to $B_i'$ in the second via the restriction of $f$. We call $M\cup_{\mathcal B, f} M$ a \emph{twisted partial double} of $M$. Note that $M\cup_{\mathcal B, f} M$ is a double cover of the manifold obtained from $M$ by gluing each $B_i$ to $B_i'$. Since finite index subgroups of residually finite groups are residually finite, we get the following Corollary to Proposition \ref{handles}.

\begin{corollary}
Let $M$ be an extended graph manifold or a pure complex-hyperbolic higher graph manifold with boundary and suppose $\pi_1(M)$ is residually finite. Then any twisted partial double of $M$ has residually finite fundamental group.
\end{corollary}

\subsection{Explicit Examples}

In this section, we provide some explicit examples of four-dimensional pure complex-hyperbolic higher graph manifolds that are also iterated partial doubles. The literature seems to contain many examples of extended graph manifolds, but on the other hand it seems to be lacking \emph{any} pure complex-hyperbolic examples. The building blocks for our constructions are the explicit examples of cusped complex-hyperbolic surfaces with minimal volume constructed in \cite{DS18}. We first present two examples, one with Euler characteristic two and the other with Euler characteristic four. Similar examples with arbitrarily large Euler characteristic can be easily manufactured. Indeed, as shown in \cite{DS17b} one can construct explicit examples of cusped complex hyperbolic surfaces saturating the whole admissible volume spectrum, that is, with Euler characteristic equal to any positive integer. By using these examples with arbitrary large Euler characteristic, one can then construct explicit pure complex-hyperbolic higher graph manifolds with zero signature, with Euler number equal to any positive even number, and with residually finite fundamental group. Finally, we present a third example where the doubling operation is twisted via a self-isomorphism.

\begin{example}\label{first}
	The first example is obtained by doubling along all its boundary components the complex hyperbolic surface with two (truncated) cusps given in Example 2 of \cite{DS18}. We denote this complex-hyperbolic surface by $S$. $S$ has two cusps with cross sections $S^1$-bundles over tori respectively of Euler number one and three. We double gluing via the identity map, and if we consider the two copies of $S$ with opposite orientations we then get an oriented $4$-manifold with signature zero and Euler characteristic two. Let $M_1$ denote this manifold, See Figure \ref{fig1}.
\end{example}

 \begin{figure}\label{fig1}
\def\svgwidth{3in}  
  \centering
\begingroup%
  \makeatletter%
  \providecommand\color[2][]{%
    \errmessage{(Inkscape) Color is used for the text in Inkscape, but the package 'color.sty' is not loaded}%
    \renewcommand\color[2][]{}%
  }%
  \providecommand\transparent[1]{%
    \errmessage{(Inkscape) Transparency is used (non-zero) for the text in Inkscape, but the package 'transparent.sty' is not loaded}%
    \renewcommand\transparent[1]{}%
  }%
  \providecommand\rotatebox[2]{#2}%
  \newcommand*\fsize{\dimexpr\f@size pt\relax}%
  \newcommand*\lineheight[1]{\fontsize{\fsize}{#1\fsize}\selectfont}%
  \ifx\svgwidth\undefined%
    \setlength{\unitlength}{473.54365792bp}%
    \ifx\svgscale\undefined%
      \relax%
    \else%
      \setlength{\unitlength}{\unitlength * \real{\svgscale}}%
    \fi%
  \else%
    \setlength{\unitlength}{\svgwidth}%
  \fi%
  \global\let\svgwidth\undefined%
  \global\let\svgscale\undefined%
  \makeatother%
  \begin{picture}(1,0.42277379)%
    \lineheight{1}%
    \setlength\tabcolsep{0pt}%
    \put(0,0){\includegraphics[width=\unitlength,page=1]{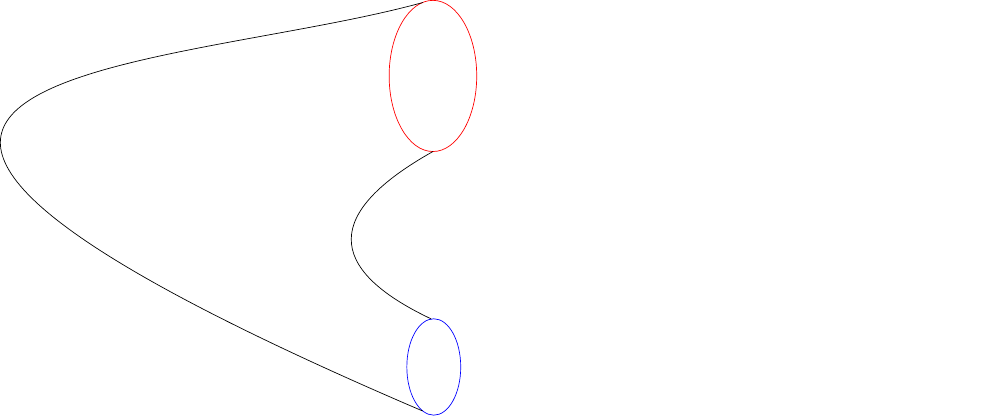}}%
    \put(0.17023785,0.21567045){\color[rgb]{0,0,0}\makebox(0,0)[lt]{\lineheight{1.25}\smash{\begin{tabular}[t]{l}$+$\end{tabular}}}}%
    \put(0,0){\includegraphics[width=\unitlength,page=2]{fig1.pdf}}%
    \put(0.74971957,0.21766183){\color[rgb]{0,0,0}\makebox(0,0)[lt]{\lineheight{1.25}\smash{\begin{tabular}[t]{l}$-$\end{tabular}}}}%
  \end{picture}%
\endgroup%
 \\
	\caption{Two copies of $S$ with opposite orientations. The red boundary components are $S^1$-bundles with Euler number 3 while the blue ones are $S^1$ bundles with Euler number 1. Adjacent boundary components are glued together via the identity map to produce $M_1$. The symbol ``+'' refers to the canonical orientation induced by the complex structure, and naturally the symbol ``--'' refers to the opposite orientation. } 
\end{figure}

\begin{example}\label{second}
	The second example is obtained by two iterated partial doubling operations on the complex hyperbolic surface with four cusps given in Example 1 of \cite{DS18} (the so-called Hirzebruch's example \cite{Hir84}). We denote this complex-hyperbolic surface by $X$. $X$ has four cusps with cross sections $S^1$-bundles over tori all with Euler number one. We double first by gluing via the identity two boundary components, and then we double again to obtain a closed manifold. Again, if we consider the two vertices with the opposite orientation we then get an oriented $4$-manifold with signature zero and Euler characteristic four which we denote by $M_2$. See Figure \ref{fig2}. 
\end{example}

 \begin{figure}\label{fig2}
\def\svgwidth{2in}  
  \centering
\begingroup%
  \makeatletter%
  \providecommand\color[2][]{%
    \errmessage{(Inkscape) Color is used for the text in Inkscape, but the package 'color.sty' is not loaded}%
    \renewcommand\color[2][]{}%
  }%
  \providecommand\transparent[1]{%
    \errmessage{(Inkscape) Transparency is used (non-zero) for the text in Inkscape, but the package 'transparent.sty' is not loaded}%
    \renewcommand\transparent[1]{}%
  }%
  \providecommand\rotatebox[2]{#2}%
  \newcommand*\fsize{\dimexpr\f@size pt\relax}%
  \newcommand*\lineheight[1]{\fontsize{\fsize}{#1\fsize}\selectfont}%
  \ifx\svgwidth\undefined%
    \setlength{\unitlength}{447.65247255bp}%
    \ifx\svgscale\undefined%
      \relax%
    \else%
      \setlength{\unitlength}{\unitlength * \real{\svgscale}}%
    \fi%
  \else%
    \setlength{\unitlength}{\svgwidth}%
  \fi%
  \global\let\svgwidth\undefined%
  \global\let\svgscale\undefined%
  \makeatother%
  \begin{picture}(1,1.21305135)%
    \lineheight{1}%
    \setlength\tabcolsep{0pt}%
    \put(0,0){\includegraphics[width=\unitlength,page=1]{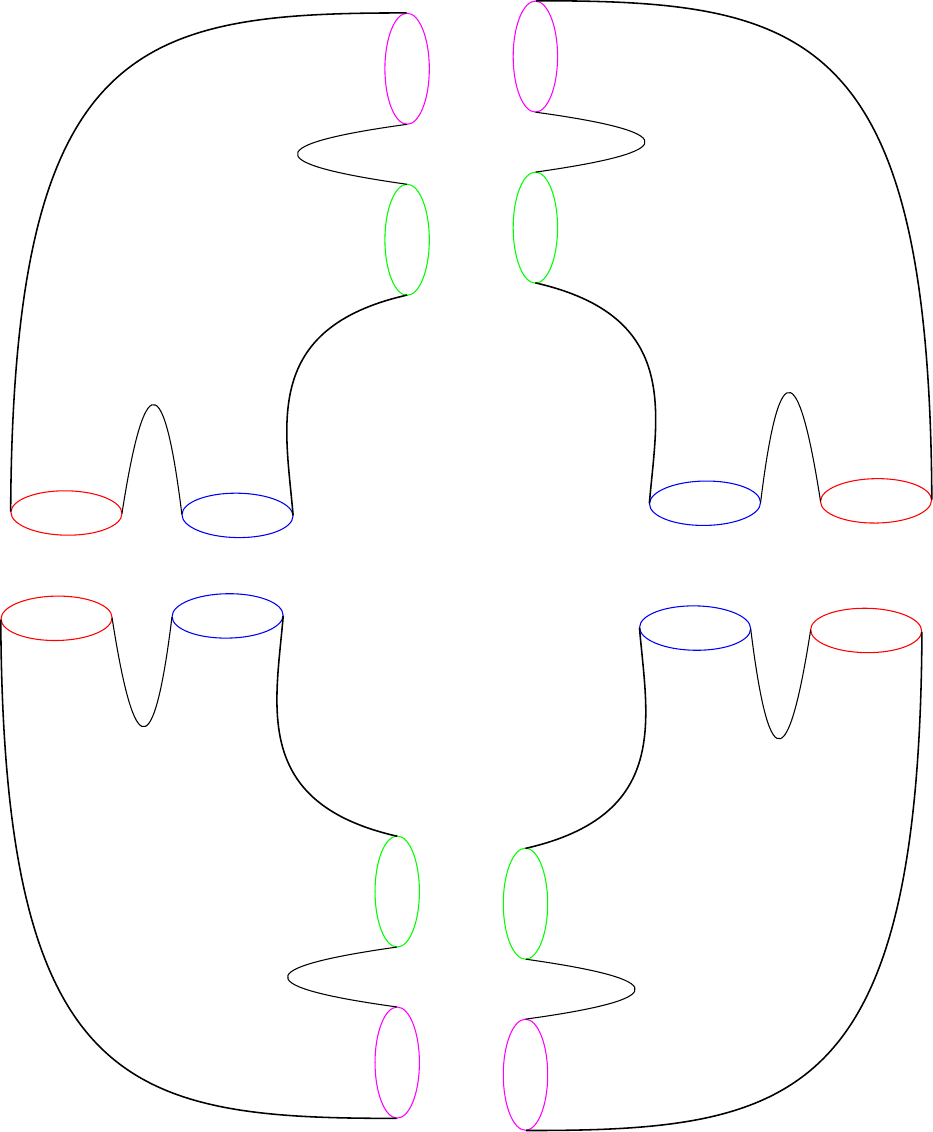}}%
    \put(0.13562575,0.25236899){\color[rgb]{0,0,0}\makebox(0,0)[lt]{\lineheight{1.25}\smash{\begin{tabular}[t]{l}$+$\end{tabular}}}}%
    \put(0.13788618,0.96666458){\color[rgb]{0,0,0}\makebox(0,0)[lt]{\lineheight{1.25}\smash{\begin{tabular}[t]{l}$-$\end{tabular}}}}%
    \put(0.76854591,0.968925){\color[rgb]{0,0,0}\makebox(0,0)[lt]{\lineheight{1.25}\smash{\begin{tabular}[t]{l}$+$\end{tabular}}}}%
    \put(0.74820206,0.22524384){\color[rgb]{0,0,0}\makebox(0,0)[lt]{\lineheight{1.25}\smash{\begin{tabular}[t]{l}$-$\end{tabular}}}}%
  \end{picture}%
\endgroup%
 \\
	\caption{$4$ copies of $X$ with orientations as shown. Adjacent boundary components are glued together via the identity map to produce $M_2$.} 
\end{figure}
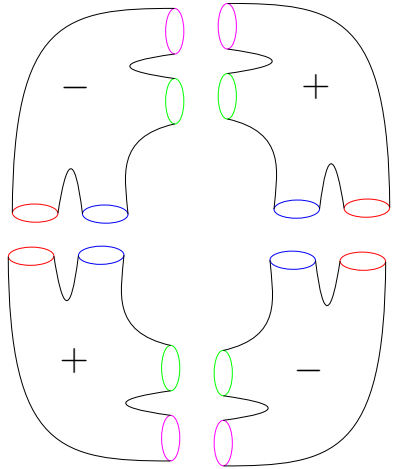

\begin{example}\label{third}
	The third example is obtained again starting with $X$ as in Example \ref{second}. In order to obtain a twisting self-isomorphism, we need to describe Hirzebruch's example in some more details. One starts with the elliptic curve $E_{\zeta}$ associated to the lattice $\IZ[1, \zeta]$  where $\zeta=e^{\frac{\pi i}{3}}$. Inside the abelian surface $A=E_{\zeta}\times E_{\zeta}$ with coordinates $(z, w)$, one considers the configuration of elliptic curves
	\begin{align}\label{configuration}
	w=0, \quad z=0, \quad w=z,\quad w=\zeta z.
	\end{align}
	The abelian surface $A$ admits a self-isomorphism $\psi: A\to A$ that preserves the configuration given in Equation \eqref{configuration} defined by setting
	\[
	\psi(z, w)=(w, \zeta z).
	\]
	More precisely, $\psi$ maps the first elliptic curve in \eqref{configuration} to the second, the second to the first, the third to the fourth, and the fourth to the third. Now the configuration of elliptic curves in Equation \eqref{configuration} intersect in the point $(0, 0)\in A$ only, and by blowing-up this point we obtain a pair $(X, D)$ where $X$ is a blown-up abelian surface and where $D$ is a divisor consisting of four smooth disjoint elliptic curves, say $D_1$, $D_2$, $D_3$, $D_4$. We then have that $\psi:A\to A$ lifts to a self-isomorphism $\bar{\psi}: (X, D)\to (X, D)$ where
	\[
	\bar{\psi}(D_1)=D_2, \quad \bar{\psi}(D_2)=D_1,\quad \bar{\psi}(D_3)=D_4, \quad \bar{\psi}(D_4)=D_3, \quad \bar{\psi}(E)=E,
	\]
	where $E$ is the exceptional divisor in $X$. Now $(X, D)$ is a smooth toroidal compactification of a complex-hyperbolic surface with Euler characteristic one and with four cusps, see \cite{DS18} for more details. The cusps are in one-to-one correspondence with the elliptic curves $D_{i}$, $i=1, 2, 3, 4$, and the complex-hyperbolic surface is biholomorphic to $X\setminus D$. Thus, $\bar{\psi}: X\setminus D\to X\setminus D$ is a biholomorphism preserving the punctured exceptional divisor which is totally geodesic in $(X\setminus D, g_{\IC})$ (see Proposition 2.1 in \cite{DS19b}).

	Now as in Example \ref{first}, $X\setminus D$ has four cusps with cross sections $S^1$-bundles over tori all with Euler number one. We can then produce a twisted double by using $\bar{\psi}$. We denote this manifold by $M_3$. Again, if we consider the two copies of $X$ with the opposite orientation we then get an oriented $4$-manifold with signature zero and Euler characteristic two. See Figure \ref{fig3}. 
\end{example}
 \begin{figure}
\def\svgwidth{3in}  
  \centering
\begingroup%
  \makeatletter%
  \providecommand\color[2][]{%
    \errmessage{(Inkscape) Color is used for the text in Inkscape, but the package 'color.sty' is not loaded}%
    \renewcommand\color[2][]{}%
  }%
  \providecommand\transparent[1]{%
    \errmessage{(Inkscape) Transparency is used (non-zero) for the text in Inkscape, but the package 'transparent.sty' is not loaded}%
    \renewcommand\transparent[1]{}%
  }%
  \providecommand\rotatebox[2]{#2}%
  \newcommand*\fsize{\dimexpr\f@size pt\relax}%
  \newcommand*\lineheight[1]{\fontsize{\fsize}{#1\fsize}\selectfont}%
  \ifx\svgwidth\undefined%
    \setlength{\unitlength}{651.64613739bp}%
    \ifx\svgscale\undefined%
      \relax%
    \else%
      \setlength{\unitlength}{\unitlength * \real{\svgscale}}%
    \fi%
  \else%
    \setlength{\unitlength}{\svgwidth}%
  \fi%
  \global\let\svgwidth\undefined%
  \global\let\svgscale\undefined%
  \makeatother%
  \begin{picture}(1,0.45548057)%
    \lineheight{1}%
    \setlength\tabcolsep{0pt}%
    \put(0,0){\includegraphics[width=\unitlength,page=1]{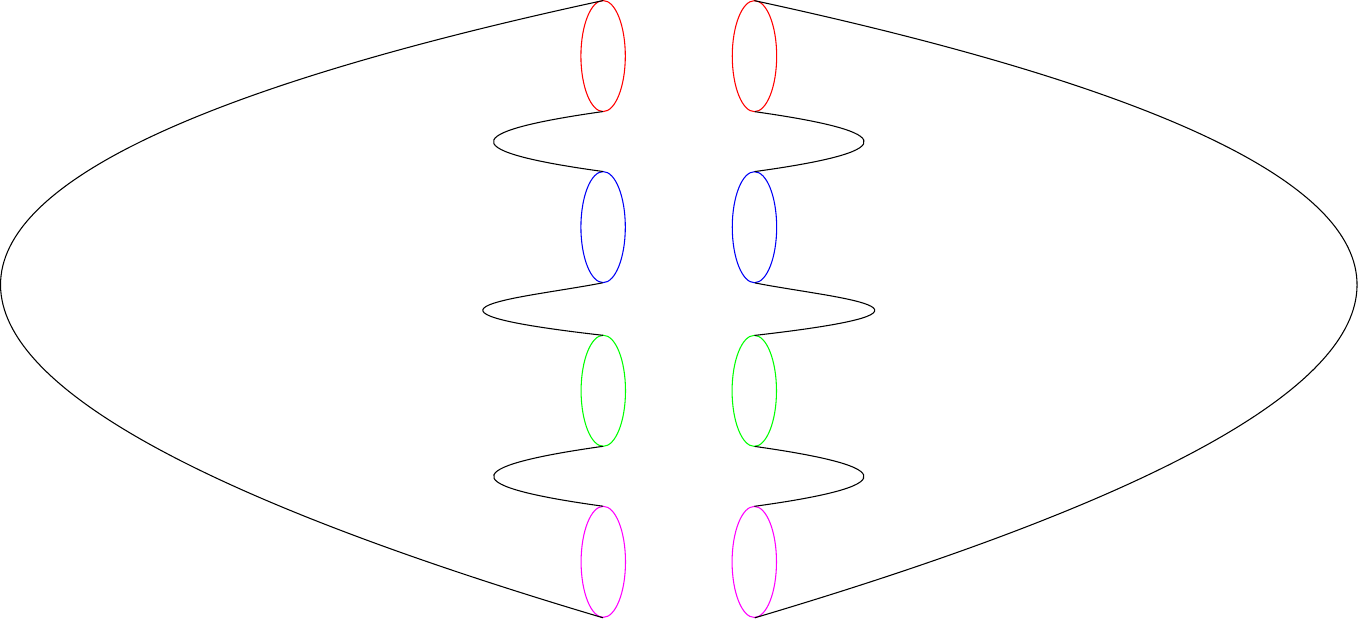}}%
    \put(0.17541024,0.23539114){\color[rgb]{0,0,0}\makebox(0,0)[lt]{\lineheight{1.25}\smash{\begin{tabular}[t]{l}$+$\end{tabular}}}}%
    \put(0.74417782,0.23429309){\color[rgb]{0,0,0}\makebox(0,0)[lt]{\lineheight{1.25}\smash{\begin{tabular}[t]{l}$-$\end{tabular}}}}%
    \put(0,0){\includegraphics[width=\unitlength,page=2]{fig3.pdf}}%
  \end{picture}%
\endgroup%
 \\
	\caption{$2$ copies of $X$ with opposite orientations. Boundary components are glued together via the restriction of $\bar{\psi}$ according to the given matching to produce $M_3$.} 
	\label{fig3}
\end{figure}

\begin{remark}
	Theorem \ref{complex} and Theorem \ref{doublerf} immediately imply that the closed $4$-manifolds constructed in Example \ref{first} and Example \ref{second} do satisfy the Singer conjecture. Recall that because of Lemma 14 in \cite{CSS19} both examples are indeed aspherical.
\end{remark}

\section{The Singer Conjecture in Dimension Three Revisited}\label{revisited}


The Singer Conjecture was proved for 3--manifolds in a celebrated paper of L\"uck and Lott. Here we show that Theorem \ref{real} can be used to give another proof. 

We first recall some 3-manifold topology; for more details, see \cite{AFW, Thu80}. A closed 3-manifold $M$ is aspherical if and only if $\pi_1(M)$ is infinite and $M$ is irreducible, that is every embedded 3-sphere bounds a 3-ball in $M$.  Passing to a double cover of $M$ if necessary, we can assume that $M$ is orientable. To any such $M$ we can apply the Geometrization Theorem, see Theorem 1.7.6 in \cite{AFW} for the version we state here.

\begin{theorem}[Geometrization Theorem (Perelman)]
Let $M$ be a closed, irreducible, orientable 3-manifold. Then $M$ has a collection of disjoint, embedded, incompressible tori $\{T_1,...,T_m\}$ such that each component of $M\setminus \{T_1,...,T_m\}$ is either Seifert-fibered or hyperbolic. 
\end{theorem}

We next use the Geometrization Theorem to show that, except in a few special cases, closed aspherical 3-manifolds are finitely covered by extended graph manifolds.



\begin{lemma}\label{3man-graph}
Let $M$ be a closed, aspherical 3-manifold. Suppose $M$ is not Seifet-fibered and not finitely covered by a torus bundle over a circle. Then $M$ is finitely covered by an extended graph manifold.
\end{lemma}
\begin{proof}
Passing to a double cover if necessary, we can assume $M$ is orientable.  By the Geometrization theorem, $M$ has a decomposition along tori into pieces which are either hyperbolic or Seifert-fibered. Some of these pieces may be twisted $I$-bundles over Klein bottles, but we can pass to a finite cover of $M$ and refine the collection of tori we are cutting along such that no such pieces occur. In particular, after passing to a finite cover of $M$ we can assume that each Seifert-fibered piece in the decomposition has a hyperbolic base orbifold (see, for example, Lemma 3.8 of \cite{GHL}). Let $P_1, ..., P_k$ denote the hyperbolic pieces and $Q_1,..,Q_l$ denote the Seifert-fibered pieces. For each $Q_i$,  as in Lemma 4.1 of \cite{Hempel} we can find a finite cover $\overline{Q}_i\cong S^1\times \Sigma_i$ of $Q_i$ where $\Sigma_i$ is a compact hyperbolic surface and each boundary component of $\overline{Q}_i$ projects homeomorphically to a boundary component of $Q$. Considering each $P_i$ as a cover of itself, these  will form a compatible collection of covers. Hence by Theorem \ref{compcovers} there is a finite cover $\overline{M}$ of $M$ such that $\overline{M}$ is a graph of spaces with tori edge spaces and each vertex space is hyperbolic or is a copy of $S^1\times \Sigma_i$ for some $i$. Hence $\overline{M}$ is an extended graph manifold.
\end{proof}

\begin{theorem}[Lott-L\"uck \cite{LL95} \& Geometrization]\label{Revisited}
The Singer conjecture holds for closed aspherical 3-manifolds.
\end{theorem}

\begin{proof}
Let $M$ be a closed, aspherical 3--manifold. If $M$ is Seifert-fibered or has a geometry modeled on torus bundle over the circle with Anosov gluing (the so-called Sol-geometry), then $MinVol(M)=0$ and $\pi_1(M)$ is residually finite. For a proof of these facts, we refer to Examples 0.2 and 1.2 in \cite{CG86}. Notice that Example 0.2 is continued at page 311 in \cite{CG86}. By Proposition \ref{collapsing}, the $L^2$--Betti numbers of $M$ vanish. Lemma \ref{3man-graph} shows that in all other cases, $M$ is finitely covered by an extended graph manifold. $\pi_1(M)$ is residually finite by \cite{Hempel} and hence the $L^2$-Betti numbers vanish by Theorem \ref{real}.  
\end{proof}

\begin{remark}
Alternatively in the proof of Theorem \ref{Revisited}, if $M$ is Seifert-fibered or has a geometric Sol-structure, then $\pi_1(M)$ has an infinite normal abelian subgroup and hence the $L^2$--Betti numbers vanish by a result of Cheeger-Gromov, see Corollary 0.6 in \cite{CG86b}. 
\end{remark}

\bibliographystyle{amsplain}

\end{document}